\documentclass[11pt]{article}
\usepackage[utf8]{inputenc}
\usepackage[T1]{fontenc}
\usepackage[frenchb,english]{babel} 
\usepackage{textcomp}
\usepackage{amsmath,amssymb}
\usepackage{amsthm}
\usepackage{lmodern}

\usepackage[a4paper,margin=2cm]{geometry}

\usepackage{graphicx}             
\usepackage{xcolor}               
\usepackage{microtype,stmaryrd}          

\usepackage{hyperref}
\hypersetup{pdfstartview=XXZ}

\usepackage{bbm}
\usepackage{mathrsfs}
\usepackage{subfig}

\def\build#1_#2^#3{\mathrel{
\mathop{\kern 0pt#1}\limits_{#2}^{#3}}}

\newtheorem{theorem}{Theorem}
\newtheorem{proposition}[theorem]{Proposition}
\newtheorem{definition}[theorem]{Definition}
\newtheorem{lemma}[theorem]{Lemma}

\def\w{\mathrm{w}}

\def\t{\mathcal{T}}

\def\f{\mathcal{F}}

\def\W{\mathcal{W}}
\def\S{\mathcal{S}}

\def\N{\mathbb{N}}

\def\D{\mathbb{D}}
\def\P{\mathbb{P}}

\def\E{\mathbb{E}}

\def\R{\mathbb{R}}
\def\z{\mathcal{Z}}

\def\n{\mathcal{N}}

\def\ve{{\varepsilon}}
\def\la{\longrightarrow}

\def\ov{\overline}

\def\dd{\mathrm{d}}
\def\wh{\widehat}
\def\wt{\widetilde}
\def\tr{\mathrm{tr}}

\def\Z{\mathbb{Z}}

\def\XX{\mathcal{X}}
\def\II{\mathrm{i}}

\def\bn{\mathbf{n}}

\def\rem{\noindent{\bf Remark. }}

\author{Jean-Fran\c cois Le Gall\footnote{Supported by the ERC Advanced Grant 740943 {\sc GeoBrown}} 
\ and Edwin Perkins\footnote{Supported by an NSERC Canada Discovery Grant}}
\title{A stochastic differential equation for local times \\
of super-Brownian motion}
\date{\small Universit\'e Paris-Saclay and University of British Columbia}

\begin{document}
\maketitle

\begin{abstract}
We show that local times of super-Brownian motion, or of Brownian motion indexed
by the Brownian tree, satisfy an explicit stochastic differential equation. 
Our proofs rely
on both excursion theory for the Brownian snake and tools from the theory of superprocesses.
\end{abstract}

\section{Introduction}

The main purpose of the present work is to derive a stochastic differential equation 
for the local times of super-Brownian motion, or equivalently for the local times of Brownian motion
indexed by the Brownian tree. Consider a super-Brownian motion whose initial value is 
a constant multiple of the Dirac measure at $0$. Informally, the local time $L^a$
at level $a\in\R$ counts how many ``particles'' visit
the point $a$. It was shown recently \cite{Mar} that, although the process $(L^a)_{a\geq 0}$
is not Markov, the pair consisting of $L^a$ and its derivative, $\dot L^a$, is a Markov process
(when $a=0$ we need to consider the right derivative at $0$). However, the transition kernel of
this Markov process is identified in \cite{Mar} in a complicated manner. Our goal here is 
to characterize this transition kernel in terms of a stochastic differential equation. There
is an obvious analogy between our main result  and the classical Ray-Knight theorems
showing that the local times of a linear Brownian motion taken at certain particular stopping 
times, and viewed as processes in the space variable, are squared Bessel processes which satisfy simple stochastic
differential equations. In the setting of the present paper, it is remarkable that the 
relevant stochastic differential equation involves the derivative of the local time.

Let us give a more precise description of our main result. On a given probability space,
we consider a super-Brownian motion $\mathbf{X}=(\mathbf{X}_t)_{t\geq 0}$
with initial value $\mathbf{X}_0=\alpha\,\delta_0$, where $\alpha>0$ is a constant. The associated
total occupation measure is defined by
$$\mathbf{Y}:=\int_0^\infty \mathbf{X}_t\,\dd t.$$
Since $\mathbf{X}$ becomes extinct a.s., the measure $\mathbf{Y}$ is finite. Sugitani \cite{Sug}
proved that the measure $\mathbf{Y}$ has a.s. a continuous density $(L^a)_{a\in\R}$,
which is even continuously differentiable on $(-\infty,0)\cup(0,\infty)$. We write 
$\dot L^a$ for the derivative of this function at $a\in\R\backslash\{0\}$. Moreover, the function $a\mapsto L^a$
has a right derivative $\dot L^{0+}$ and a left derivative $\dot L^{0-}$ at $0$, and, by convention, 
we set $\dot L^0=\dot L^{0+}$. 
In order to state our result, let $U=(U_t)_{t\geq 0}$ be a stable L\'evy process with  index $3/2$ and no negative jumps.
The distribution of $U$  is characterized by specifying its Laplace exponent $\psi(\lambda)=\sqrt{2/3}\,\lambda^{3/2}$
(see Section \ref{subsec:bridge}). For every $t>0$, let $(p_t(x))_{x\in\R}$ be the continuous
density of $U_t$, which is determined by its Fourier transform
$$\int_\R e^{\II u x}\,p_t(x)\,\dd x = \exp(-c_0t\, |u|^{3/2}\,(1+\II\,\mathrm{sgn}(u))),$$
where $c_0=1/\sqrt{3}$ and $\mathrm{sgn}(u)=\mathbf{1}_{\{u>0\}}-\mathbf{1}_{\{u<0\}}$. Then $x\mapsto p_t(x)=t^{-2/3}p_1(xt^{-2/3})$ is strictly positive, infinitely differentiable and 
has bounded derivatives for each $t$ (see Ch. 2 of \cite{Zol} for these and other properties of stable densities).  Write $p'_t(x)$ for the derivative of this function.

\begin{theorem}
\label{main-th}
For every $y\in\R$, set $g(0,y)=0$ and, for every $t>0$,
$$g(t,y)=8 t\,\frac{p'_t(y)}{p_t(y)}.$$
Then 
$$\int_0^\infty |g\Bigl(L^y,\frac{1}{2}\dot L^y\Bigr)|\,\dd y<\infty, \quad\hbox{a.s.}$$
and the pair $(L^x,\dot L^x)_{x\geq 0}$ satisfies the two-dimensional stochastic differential equation 
\begin{equation}\label{main-SDE}
\begin{aligned}
\dot L^x&=\dot L^0 + 4\int_0^x \sqrt{{L^y}}\,\dd B_y + \int_0^x g\Bigl(L^y,\frac{1}{2}\dot L^y\Bigr)\,\dd y\\
L^x&=L^0+\int_0^x\dot L^y\,\dd y,
\end{aligned}
\end{equation}
where $B$ is a linear Brownian motion. Moreover if $R=\inf\{x\ge 0:L^x=0\}$, then $(L^x,\dot L^x)$ is the pathwise unique solution to \eqref{main-SDE} which satisfies $(L^x,\dot L^x)=(L^{x\wedge R},\dot L^{x\wedge R})$ for all $x\ge 0$ a.s. 
\end{theorem}

\rem The fact that the local time satisfies the last property stated in the Theorem follows from Theorem~1.7 in \cite{MP} where it is shown that
if $R$ is as above and $G=\sup\{x\le 0:L^x=0\}$, then
\begin{equation}\label{Lsupp}
-\infty<G<0<R<\infty\ \text{ and }\{x\in \R:L^x>0\}=(G,R)\ a.s.
\end{equation}
Strictly speaking, in order to write equation \eqref{main-SDE}, it may be necessary to enlarge the underlying probability space.
The point is that the Brownian motion $B$
will be determined from the pair $(L^x,\dot L^x)_{x\geq 0}$ only up to the ``time'' $R$ 
(for $x>R$ we have $L^x=\dot L^x=0$). So a more precise statement would be the existence 
of an enlarged probability space $(\Omega, \f,\P)$ equipped with a filtration $(\f_t)_{t\geq 0}$ and an $(\f_t)$-Brownian motion
$B$ such that $(L^t,\dot L^t)_{t\geq 0}$ is adapted to the filtration $(\f_t)_{t\geq 0}$ and \eqref{main-SDE} holds (see
the proof in Section \ref{sec:SDE}).

\smallskip

Interestingly, the functions $p_t$ and $p'_t$ have explicit expressions in terms of the classical Airy function $\mathrm{Ai}$
and its derivative $\mathrm{Ai}'$. In fact, $x\to p_t(-x)$ is called the Airy map distribution in \cite{FS}. For every $t>0$ and $x\in\R$, we have
$$p_t(x)= 6^{-1/3}\,t^{-2/3} \,\mathcal{A}(6^{-1/3}t^{-2/3}x),$$
where
$$\mathcal{A}(x)=-2\,e^{2x^3/3}\Big( x\mathrm{Ai}(x^2) + \mathrm{Ai}'(x^2)\Big).$$
See \cite[Section IX.11]{FS}, or \cite{CC} and the references therein, and note that our choice of $p_t$ differs from that in \cite{CC} by a scaling constant. It follows that 
$$g(t,x)=8\times 6^{-1/3}\,t^{1/3} \,\frac{\mathcal{A}'}{\mathcal{A}}( 6^{-1/3}\,t^{-2/3} x),$$
with (the Airy equation $\mathrm{Ai}''(x)=x\mathrm{Ai}(x)$ helps here)
\begin{equation}\label{Airatio}\frac{\mathcal{A}'}{\mathcal{A}}(x)= 4x^2 + \frac{\mathrm{Ai}(x^2)}{x\mathrm{Ai}(x^2) + \mathrm{Ai}'(x^2)}.
\end{equation}
One useful application of this representation and known asymptotics for $\mathrm{Ai}$ and $\mathrm{Ai}'$ (see p. 448 of \cite{AS}) is that
\begin{equation}\label{rtpt}\frac{p_1'}{p_1}(y)=6^{-1/3}\,\frac{\mathcal{A}'}{\mathcal{A}}(6^{-1/3}y)= -\frac{5}{2y}+O\Bigl(\frac{1}{y^4}\Bigr)\ \text{ as }y\to+\infty,
\end{equation}
and so 
\begin{equation}\label{gbnd1}
\text{for all }y_0\in\R,\ \sup_{y\ge y_0}\Bigl|\frac{p_1'}{p_1}(y)\Bigr|=C(y_0)<\infty.
\end{equation}
%\medskip

We can reformulate our theorem in terms of the model called Brownian motion indexed by the Brownian tree.
Here the Brownian tree $\t$ is a ``free'' version of Aldous' Continuum Random Tree \cite{Ald} and may be defined as the tree coded by a Brownian excursion under the ($\sigma$-finite) It\^o measure. Points of 
$\t$ are assigned ``Brownian labels'' $(V_u)_{u\in\t}$, in such a way that the label of the root is $0$ and labels evolve like linear Brownian
motion along the line segments of the tree. It is convenient to assume that both the tree $\t$ 
and the labels $(V_u)_{u\in\t}$ are defined on the canonical space of snake trajectories under the
``excursion measure'' $\N_0$  (see Section \ref{sec:preli} below for a more precise presentation). If $\mathrm{Vol}$
denotes the volume measure on the tree $\t$, we are interested in the total occupation measure, which is the
finite measure $\mathcal{Y}$ on $\R$ defined by 
\begin{equation}\label{Ydefn}\mathcal{Y}(f)=\int_\t f(V_u)\,\mathrm{Vol}(\dd u),
\end{equation}
for every nonnegative Borel function $f$ on $\R$. The measure $\mathcal{Y}$ has a continuously differentiable density
$(\ell^x)_{x\in\R}$ with respect to Lebesgue measure on $\R$, and we write $(\dot\ell^x)_{x\in\R}$ for its derivative.
We can then state an analog of Theorem \ref{main-th}. There is a technical difficulty due to the fact that 
$\N_0$ is an infinite measure, and for this reason we need to make an appropriate conditioning.

\begin{theorem}
\label{main2}
Let $\delta>0$, and consider the probability measure $\N^{(\delta)}_0:=\N_0(\cdot\mid \ell^0>\delta)$. Then, 
$$\int_0^\infty |g\Bigl(\ell^y,\frac{1}{2}\dot \ell^y\Bigr)|\,\dd y<\infty, \quad\N^{(\delta)}_0\hbox{\ a.s.}$$
and, under $\N^{(\delta)}_0$, the pair $(\ell^x,\dot \ell^x)_{x\geq 0}$ satisfies the two-dimensional stochastic differential equation 
\begin{align*}\dot \ell^x&=\dot \ell^0 + 4\int_0^x \sqrt{{\ell^y}}\,\dd \beta_y + \int_0^x g\Bigl(\ell^y,\frac{1}{2}\dot \ell^y\Bigr)\,\dd y\\
\ell^x&=\ell^0+\int_0^x\dot \ell^y\, \dd y,
\end{align*}
where $\beta$ is a linear Brownian motion. Moreover if $\rho=\inf\{x\ge 0:\ell^x=0\}$, then $(\ell^x,\dot \ell^x)$ is the pathwise unique solution to the above equation which satisfies $(\ell^x,\dot \ell^x)=(\ell^{x\wedge \rho},\dot \ell^{x\wedge \rho})$ for all $x\ge 0$ a.s. 

\end{theorem}

In the language of superprocesses, Theorem \ref{main2} corresponds to a version of Theorem \ref{main-th} under the
so-called canonical measure. In what follows, we will only deal with Theorem \ref{main-th}. Theorem \ref{main2}
then follows since it is shown in \cite{Mar} that the process $(\ell^t,\dot \ell^t)_{t\geq 0}$ is Markov 
with the same transition kernels as the process $(L^t,\dot L^t)_{t\geq 0}$ considered in Theorem \ref{main-th} (the pathwise uniqueness in either equation will follow easily from a classical result for locally Lipschitz coefficients).
Still the formulation of Theorem \ref{main2} is useful to understand our approach, as we will rely on the Brownian
snake representation of super-Brownian motion, which involves considering a Poisson collection of Brownian trees
equipped with Brownian labels.  The same remark as for Theorem \ref{main-th} applies also to Theorem \ref{main2} (see Theorem~1.4 of \cite{Hon} for the analogue of \eqref{Lsupp}). 

One motivation for deriving a stochastic differential equation for $(L^x,\dot L^x)$ is to allow one access to the tools of stochastic analysis for 
a more detailed analysis of these processes.  To this end, we use a transformation of the state space  and a random time change to effectively transform the solution to \eqref{main-SDE} into an explicit one-dimensional diffusion which can be studied in detail, and from which one can 
reconstruct $(L^x,\dot L^x)$ (see Propositions \ref{simple-SDE} and \ref{thm:1ddiffusion}). The diffusion will be a time change of $\dot L^x/(L^x)^{2/3}$ and is the unique solution of \eqref{SDE4} below. 

Our proofs depend on both the excursion theory for the Brownian snake \cite{ALG} and tools coming from the
theory of superprocesses \cite{Per, Hon}. Excursion theory for the Brownian snake was the key ingredient 
for getting the Markov property of the process $(L^x,\dot L^x)_{x\geq 0}$ in \cite{Mar}. The transition 
kernel of this process was described in terms of the ``positive excursion measures'' $\N^{*,z}_0$, which roughly speaking
give the distribution of the labeled tree $(\t,(V_u)_{u\in\t})$ conditioned to have only nonnegative labels, with a parameter 
$z>0$ that in some sense prescribes how many points $u$ of $\t$ have the label zero. Local times 
still make sense under the measures $\N^{*,z}_0$ and, for every $h>0$, one can compute 
the expected value of the derivative of local time at level $h$  in the form 
$$\N^{*,z}_0(\dot \ell^h)= z\,\gamma\Big(\frac{3h}{2z^2}\Big),$$
where the function $\gamma$ has an explicit expression in terms of the (complementary) error function $\mathrm{erfc}$
(Proposition \ref{expected-derivative}). For every $a\geq 0$ and $t>0$, $y\in\R$, excursion
theory then leads to the formula
$$\E\Big[\dot L^{a+h}\,\Big|\, L^a=t,\frac{1}{2}\dot L^a=y\Big]=\E\Big[\sum_{j=1}^\infty Z_j\,\gamma\Big(\frac{3Z_j}{2h^2}\Big)\Big],
$$
where $(Z_j)_{j\geq 1}$ are the jumps of the bridge from $0$ to $y$ in time $t$
associated with the L\'evy process $U$ (Proposition \ref{tech-incre}) and listed in decreasing order. The precise justification of the 
formulas of the last two displays requires certain bounds on moments of the derivatives of local time
(Lemmas \ref{moment-deri} and \ref{lem-tech}).
We obtain these bounds via a stochastic integral representation of the derivative $\dot L^x$ 
in terms of the martingale measure associated with $\mathbf{X}$, which is due to Hong \cite{Hon}. 
Here the use of these techniques from the theory of superprocesses is crucial since the excursion
measures $\N^{*,z}_0$ do not seem to provide a tractable setting for a direct derivation of the required bounds. 

It turns out that one can explicitly compute 
the right-hand side of the last display in terms of an integral involving the density $p_t$ (Proposition \ref{increment-derivative})
and it is then an easy matter to obtain
$$\lim_{h\to 0} \frac{1}{h}\, \E\Big[\dot L^{a+h}-\dot L^a\,\Big|\, L^a=t,\frac{1}{2}\dot L^a=y\Big]
= 8\,t\, \frac{p'_t(y)}{p_t(y)}=g(t,y).$$
From this, one can infer that, for every $\ve>0$, the process
$$M^\ve_x:=\dot L^{x\wedge S_\ve}-\dot L^0-\int_0^{x\wedge S_\ve} g(L^y,\frac{1}{2}\dot L^y)\,\dd y$$
is a local martingale, where we have written $S_\ve:=\inf\{x\geq 0:L^x\leq \ve\}$. At that point, we again use
the stochastic integral representation of Hong \cite{Hon}, from
which we can deduce that the quadratic variation of $M^\ve_x$ is $16\int_0^{x\wedge S_\ve} L^y\,\dd y$. 
Although there are some additional technicalites to handle, often due to the unboundedness of $g(L^y,\dot L^y/2)$, we then can use 
standard tools of stochastic calculus to derive the
stochastic differential equation \eqref{main-SDE}. 

We note that the recent paper of Chapuy and Marckert \cite{CM} addresses similar questions 
for the model called ISE (integrated super-Brownian excursion). This model, which was
introduced by Aldous \cite{Al2}, corresponds to conditioning
the Brownian tree $\t$ to have total volume equal to $1$. Under this conditioning, local times are still
well defined and continuously differentiable. On the basis of
discrete approximations, \cite{CM} conjectures a stochastic differential equation
for local times of ISE, which is similar to \eqref{main-SDE} but  with
a more complicated drift term involving also the integrals $\int_{-\infty}^x L^y\,\dd y$ --- the reason 
why these integrals appear is of course the special conditioning which forces $\int_{-\infty}^\infty L^y\,\dd y=1$.
It is likely that Theorem \ref{main2} can be used to also derive a stochastic differential equation 
for local times of ISE, but we do not pursue this matter here. 

The paper is organized as follows. Section \ref{sec:preli} gathers a number of preliminaries. In particular,
we introduce the positive excursion measures $\N^{*,z}_0$, and we recall the main result of the excursion theory of \cite{ALG}.
In Section \ref{sec:super}, we briefly recall the Brownian snake construction of the super-Brownian motion $\mathbf{X}$, and
we state a key result from \cite{Mar} giving the conditional distribution of the collection of ``excursions'' of $\mathbf{X}$ above
a level $a\geq 0$ knowing $(L^x,\dot L^x)_{x\leq a}$ (Proposition \ref{key-ingre}). This conditional distribution knowing 
$L^a=t$ and $\dot L^a=y$ is given in terms of the measures $\N^{*,z}_0$
and the collection of jumps of the L\'evy bridge from $0$ to $y$ in time $t$. In Section \ref{est-moment}, we rely on
Hong's representation to derive our estimates on moments of the increments of $\dot L^x$, and then 
to evaluate the quadratic variation of this process. Section \ref{sec:exp-incre-deriv} is devoted to the calculation
of the conditonal expected value of $\dot L^{a+h}-\dot L^a$ knowing $L^a=t$ and $\dot L^a=y$. Finally, Section \ref{sec:SDE}
gives the proof of Theorem \ref{main-th} and also establishes the connection between $(L^x,\dot L^x)$ and the simple diffusion in \eqref{SDE4}.

\section{Preliminaries}
\label{sec:preli}

\subsection{Snake trajectories}
\label{sna-tra}

The proof of our main result relies in part on the Brownian snake representation 
of super-Brownian motion. We start by recalling the formalism of
snake trajectories, referring to \cite{ALG} for more details.
A (one-dimensional) finite path $\w$ is a continuous mapping $\w:[0,\zeta]\la\R$, where $\zeta=\zeta_{(\w)}\in[0,\infty)$ is called the lifetime of $\w$. The space $\W$ of all finite paths is a Polish space when equipped with the
distance
$$d_\W(\w,\w')=|\zeta_{(\w)}-\zeta_{(\w')}|+\sup_{t\geq 0}|\w(t\wedge
\zeta_{(\w)})-\w'(t\wedge\zeta_{(\w')})|.$$
The endpoint or tip of the path $\w$ is denoted by $\wh \w=\w(\zeta_{(\w)})$.
For every $x\in\R$, we set $\W_x=\{\w\in\W:\w(0)=x\}$. The trivial element of $\W_x$ 
with zero lifetime is identified with the point $x$ of $\R$. 

\begin{definition}
\label{def:snakepaths}
Let $x\in \R$. A snake trajectory with initial point $x$ is a continuous mapping $s\mapsto \omega_s$
from $\R_+$ into $\W_x$ 
which satisfies the following two properties:
\begin{enumerate}
\item[\rm(i)] We have $\omega_0=x$ and the number $\sigma(\omega):=\sup\{s\geq 0: \omega_s\not =x\}$,
called the duration of the snake trajectory $\omega$,
is finite (by convention $\sigma(\omega)=0$ if $\omega_s=x$ for every $s\geq 0$). 
\item[\rm(ii)] (Snake property) For every $0\leq s\leq s'$, we have
$\omega_s(t)=\omega_{s'}(t)$ for every $t\in[0,\displaystyle{\min_{s\leq r\leq s'}} \zeta_{(\omega_r)}]$.
\end{enumerate} 
\end{definition}

We will write $\S_x$ for the set of all snake trajectories
with initial point $x$, and $\S$ for the union of the sets $\S_x$ for all $x\in\R$. If $\omega\in \S$, we often write $W_s(\omega)=\omega_s$ and $\zeta_s(\omega)=\zeta_{(\omega_s)}$
for every $s\geq 0$, and we omit $\omega$ in the notation. 
The sets $\S$ and $\S_x$ are equipped with the distance
$$d_{\S}(\omega,\omega')= |\sigma(\omega)-\sigma(\omega')|+ \sup_{s\geq 0} \,d_\W(W_s(\omega),W_{s}(\omega')).$$
For $\omega\in\S_x$ and $a\in\R$, we will use the obvious notation $\omega + a\in\S_{x+a}$ for the translated snake trajectory. 
It is easy to verify \cite[Proposition 8]{ALG} that a snake trajectory $\omega$ is determined by the two functions $s\mapsto \zeta_s(\omega)$ and $s\mapsto \wh W_s(\omega)$ (the latter is 
sometimes called the tip function).

Let $\omega\in \S$ be a snake trajectory and $\sigma=\sigma(\omega)$. We define a pseudo-distance on $[0,\sigma]^2$
by setting
$$d_\zeta(s,s')= \zeta_s+\zeta_{s'}-2 \min_{s\wedge s'\leq r\leq s\vee s'} \zeta_r.$$
We then consider the associated equivalence relation $s\sim s'$ if and only if $d_\zeta(s,s')=0$ (or equivalently 
$\zeta_s=\zeta_{s'}= \min_{s\wedge s'\leq r\leq s\vee s'} \zeta_r$), and the quotient space $\t(\omega):=[0,\sigma]/\!\sim$ ,
which is equipped with the distance induced by $d_\zeta$. The metric space $(\t(\omega),d_\zeta)$ is 
a compact $\R$-tree called the {\it genealogical tree} of the snake trajectory $\omega$ (we refer to
\cite{probasur} for more information about the
coding of $\R$-trees by continuous functions). 
Let $p_{(\omega)}:[0,\sigma]\la \t(\omega)$ stand
for the canonical projection. By convention, the tree $\t=\t(\omega)$ is rooted at the point
$\rho_{(\omega)}:=p_{(\omega)}(0)=p_{(\omega)}(\sigma)$, and the volume measure $\mathrm{Vol}(\cdot)$ on $\t$ is defined as the pushforward of
Lebesgue measure on $[0,\sigma]$ under $p_{(\omega)}$.  As usual, for $u,v\in\t$, we say that $u$ is an ancestor
of $v$, or $v$ is a descendant of $u$, if $u$ belongs to the line segment from $\rho_{(\omega)}$ to $v$ in $\t$. 

The snake property shows that the condition $p_{(\omega)}(s)=p_{(\omega)}(s')$ implies that 
$W_s(\omega)=W_{s'}(\omega)$. So the mapping $s\mapsto W_s(\omega)$ can be viewed as defined on the quotient space $\t$.
For $u\in \t$, we set $V_u:=\wh W_s(\omega)$, for any $s\in[0,\sigma]$ such that $u=p_{(\omega)}(s)$. We interpret $V_u$ as a ``label'' assigned to the ``vertex'' $u$ of $\t$,
and each path $W_s$ records the labels along the line segment from $\rho_{(\omega)}$ to $p_{(\omega)}(s)$ in $\t$.

We will use the notation 
\begin{align*}
W^*&:=\max\{W_s(t): s\geq 0,t\in[0,\zeta_s]\}=\max\{\wh W_s:0\leq s\leq \sigma\}= \max\{V_u:u\in\t\},\\
W_*&:=\min\{W_s(t): s\geq 0,t\in[0,\zeta_s]\}= \min\{\wh W_s:0\leq s\leq \sigma\}=\min\{V_u:u\in\t\},
\end{align*}
and we also let $\mathcal{Y}=\mathcal{Y}_{(\omega)}$ be the  occupation measure of $\omega$, which is the finite measure on $\R$ defined by setting
\begin{equation}
\label{occu-measure}
\mathcal{Y}(f)= \int_0^{\sigma} f(\wh W_s)\,\dd s=\int_{\t} f(V_u)\,\mathrm{Vol}(\dd u),
\end{equation}
for any Borel function $f:\R\la \R_+$. Trivially, $\mathcal{Y}$ is supported on $[W_*,W^*]$. 

%We now introduce two important operations on snake trajectories. The first one 
%is the re-rooting operation (see \cite[Section 2.2]{ALG}). Let $\omega\in \S_x$ and
%$r\in[0,\sigma(\omega)]$. Then $\omega^{[r]}$ is the snake trajectory in $\S_{\hat\omega_r}$ such that
%$\sigma(\omega^{[r]})=\sigma(\omega)$ and for every $s\in [0,\sigma(\omega)]$,
%\begin{align*}
%\zeta_s(\omega^{[r]})&= d_\zeta(r,r\oplus s),\\
%\wh W_s(\omega^{[r]})&= \wh W_{r\oplus s}(\omega),
%\end{align*}
%where we use the notation $r\oplus s=r+s$ if $r+s\leq \sigma$, and $r\oplus s=r+s-\sigma$ otherwise. 
%By a remark following the definition of snake trajectories, these prescriptions completely determine $\omega^{[r]}$.
%
%%Let us write $\zeta_s^{[r]}(\omega)=\zeta_s(\omega^{[r]})$.% and $W^{[r]}_s(\omega)=W_s(\omega^{[r]})$. 
%The genealogical tree $\t(\omega^{[r]})$ is then interpreted as the tree $\t(\omega)$ re-rooted at the vertex $p_{(\omega)}(r)$: More precisely,
%the mapping $s\mapsto r\oplus s$ induces an isometry from $\t(\omega^{[r]})$
%onto $\t(\omega)$, which maps the root of $\t(\omega^{[r]})$ to $p_{(\omega)}(r)$. Furthermore, the vertices
%of $\t(\omega^{[r]})$ receive the same labels as in $\t(\omega)$.  

We next introduce the truncation of snake trajectories. For any $\w\in\W_x$ and $y\in\R$, we set
$$\tau_y(\w):=\inf\{t\in(0,\zeta_{(\w)}]: \w(t)=y\}\,,$$
with the usual convention $\inf\varnothing =\infty$.
Then if $\omega\in \S_x$ and $y\in \R$, we set, for every $s\geq 0$,
$$\nu_s(\omega):=\inf\Big\{t\geq 0:\int_0^t \mathrm{d}u\,\mathbf{1}_{\{\zeta_{(\omega_u)}\leq\tau_y(\omega_u)\}}>s\Big\}$$
(note that the condition $\zeta_{(\omega_u)}\leq\tau_y(\omega_u)$ holds if and only if $\tau_y(\omega_u)=\infty$ or $\tau_y(\omega_u)=\zeta_{(\omega_u)}$).
Then, setting $\omega'_s=\omega_{\nu_s(\omega)}$ for every $s\geq 0$ defines an element $\omega'$ of $\S_x$,
which will be denoted by  ${\rm tr}_y(\omega)$ and called the truncation of $\omega$ at $y$
(see \cite[Proposition 10]{ALG}). The effect of the time 
change $\nu_s(\omega)$ is to ``eliminate'' those paths $\omega_s$ that hit $y$ and then survive for a positive
amount of time. The genealogical tree of 
${\rm tr}_y(\omega)$ is canonically and isometrically identified 
with the closed subset of $\t(\omega)$ consisting of all $u$ such that
$V_v(\omega)\not =y$ for every strict ancestor $v$ of $u$
(different from $\rho_{(\omega)}$ when $y=x$). 

Finally, for $\omega\in\S_x$ and $y\in\R\backslash\{x\}$, we  define the excursions of $\omega$ away from $y$. In contrast with the truncation
${\rm tr}_y(\omega)$, these excursions now account for the
paths $\omega_s$ that hit $y$ and survive for a positive
amount of time. More precisely, let $(\alpha_j,\beta_j)$, $j\in J$, be the connected components of the open set
$\{s\in[0,\sigma]:\tau_y(\omega_s)<\zeta_{(\omega_s)}\}$
(note that the indexing set $J$ may be empty).
We notice that $\omega_{\alpha_j}=\omega_{\beta_j}$ for every $j\in J$, by the snake property, and $\wh\omega_{\alpha_j}=y$.
For every $j\in J$, we define a snake trajectory $\omega^j\in\S_y$ by setting
$$\omega^j_{s}(t):=\omega_{(\alpha_j+s)\wedge\beta_j}(\zeta_{(\omega_{\alpha_j})}+t)\;,\hbox{ for }0\leq t\leq \zeta_{(\omega^j_s)}
:=\zeta_{(\omega_{(\alpha_j+s)\wedge\beta_j})}-\zeta_{(\omega_{\alpha_j})}\hbox{ and } s\geq 0.$$
We say that $\omega^j$, $j\in J$, are the excursions of $\omega$ away from $y$. 

\subsection{The Brownian snake excursion measure}
\label{sna-mea}

Let $x\in\R$. The Brownian snake excursion 
measure $\N_x$ is the $\sigma$-finite measure on $\S_x$ that is characterized by the following two properties: Under $\N_x$,
\begin{enumerate}
\item[(i)] the distribution of the lifetime function $(\zeta_s)_{s\geq 0}$ is the It\^o 
measure of positive excursions of linear Brownian motion, normalized so that, for every $\ve>0$,
$$\N_x\Big(\sup_{s\geq 0} \zeta_s >\ve\Big)=\frac{1}{2\ve};$$
\item[(ii)] conditionally on $(\zeta_s)_{s\geq 0}$, the tip function $(\wh W_s)_{s\geq 0}$ is
a Gaussian process with mean $x$ and covariance function 
$$K(s,s')= \min_{s\wedge s'\leq r\leq s\vee s'} \zeta_r.$$
\end{enumerate}
Conditionally on the lifetime process $(\zeta_s)_{s\geq 0}$, each path $W_r$ is a linear Brownian path started from $x$ with
lifetime $\zeta_r$. When $r$ varies, the evolution of the path $W_r$ is described informally as follows. When $\zeta_r$ decreases, the path $W_r$
is ``erased'' from its tip, and when $\zeta_r$ increases, the path $W_r$ is ``extended'' by adding little pieces of Brownian motion at its tip.
The measure $\N_x$ can be interpreted as the excursion measure away from $x$ for the 
Markov process in $\W_x$ called the (one-dimensional) Brownian snake. We refer to 
\cite{Zurich} for a detailed study of the Brownian snake with a more general
underlying spatial motion.

For every $r>0$, we have
$$\N_x(W^*>x+r)=\N_x(W_*<x-r)={\displaystyle \frac{3}{2r^2}}$$ (see e.g. \cite[Section VI.1]{Zurich}). 
In particular, $\N_x(y\in[W_*,W^*])<\infty$ if $y\not =x$.

The following scaling property is often useful. For $\lambda>0$, for every 
$\omega\in \S_x$, we define $\theta_\lambda(\omega)\in \S_{x\sqrt{\lambda}}$
by $\theta_\lambda(\omega)=\omega'$, with
\begin{equation}\label{scalereln}\omega'_s(t):= \sqrt{\lambda}\,\omega_{s/\lambda^2}(t/\lambda)\,,\ 
\hbox{for } s\geq 0\hbox{ and }0\leq t\leq \zeta'_s:=\lambda\zeta_{s/\lambda^2}.
\end{equation}
Then $\theta_\lambda(\N_x)= \lambda\, \N_{x\sqrt{\lambda}}$. 

Let us introduce the local times, $(\ell^y)_{y\in\R}$, under $\N_x$. The next proposition follows from \cite{CM} (a slightly weaker statement had been obtained in \cite{BMJ}), and 
 is also closely related to the results of \cite{Sug} concerning super-Brownian motion.

\begin{proposition}
\label{exist-LT}
Let $x\in\R$. Then, $\N_x(\dd \omega)$ a.e. the occupation measure $\mathcal{Y}_{(\omega)}$ has a continuously differentiable density
with respect to Lebesgue measure. This density is denoted by $(\ell^y(\omega))_{y\in\R}$ and its derivative is
denoted by $(\dot\ell^y(\omega))_{y\in\R}$
\end{proposition}

We now introduce exit measures. We argue under $\N_x$, and fix $y\in\R\backslash\{x\}$. 
One shows that the limit
\begin{equation}
\label{formu-exit}
\z_y:=\lim_{\ve \to 0} \frac{1}{\ve} \int_0^\sigma \dd s\,\mathbf{1}_{\{\tau_y(W_s)\leq \zeta_s\leq\tau_y(W_s)+\ve\}}
\end{equation}
exists $\N_x$ a.e., and $\z_y$ is called the exit measure from $(y,\infty)$ (if $x>y$) or from $(-\infty,y)$ (if $y>x$). 
Roughly speaking, $\z_y$ counts how many paths $W_s$ hit $y$ and are stopped at that moment. 
The definition of $\z_y$
is a particular case of the theory of exit measures, see \cite[Chapter V]{Zurich}. We have 
$\z_y>0$ if and only if $y\in[W_*,W^*]$, $\N_x$ a.e. (recall $y$ is fixed).

Let us recall the special Markov property of the Brownian snake under $\N_0$ (see, for example, the appendix
of \cite{subor}).  

\begin{proposition}[Special Markov property]
\label{SMP}
Let $x\in \R$ and $y\in\R\backslash\{x\}$. Under the measure $\N_x(\dd \omega)$, 
let $\omega^j$, $j\in J$, be the excursions of $\omega$ away from $y$ and consider the point measure
$$\n_y=\sum_{j\in J} \delta_{\omega^j}.$$
Then, under the probability measure $\N_x(\dd\omega\,|\, y\in[W_*,W^*])$ and conditionally on $\z_y$, the point measure
$\n_y$ is Poisson with intensity $\z_y\,\N_y(\cdot)$ and is independent of $\mathrm{tr}_y(\omega)$.
\end{proposition}

We now introduce a process called the exit measure process at a point, which will play an important 
role in the excursion theory discussed below. Let $x\in \R$ and argue under the excursion  measure $\N_x$.
Also fix another point $y\in\R$ (which may be equal to $x$). Since, conditionally on $\zeta_s$, $W_s$ is 
just a Brownian path with lifetime $\zeta_s$, we can make sense of its local time at level $y$,
which we denote by $\mathcal{L}^y(W_s)=(\mathcal{L}^y_t(W_s))_{0\leq t\leq \zeta_s}$, and the mapping
$s\mapsto \mathcal{L}^y(W_s)$, with values in $(\W,d_\W)$, is continuous (note that $(W_s,\mathcal{L}^y(W_s))$ can be viewed 
as the Brownian snake whose spatial motion is the pair formed by Brownian motion and its local time at $y$). Then, for every $r\geq 0$ and $s\in[0,\sigma]$, set
$$\eta^y_r(W_s)=\inf\{t\in[0,\zeta_s]: \mathcal{L}^y_t(W_s)\geq r\},$$
with the usual convention $\inf\varnothing=\infty$. From the general theory of exit measures \cite[Chapter V]{Zurich},
we get, for every $r>0$, the existence of the almost sure limit
$$\mathcal{X}^y_r=\lim_{\ve\to 0} \frac{1}{\ve}\int_0^\sigma \dd s\, \mathbf{1}_{\{\eta^y_r(W_s)\leq \zeta_s\leq \eta^y_r(W_s)+\ve\}}.$$
Roughly speaking, $\mathcal{X}^y_r$ measures the
``quantity'' of paths $W_s$ that end at $y$ after having accumulated a
local time at $y$ exactly equal to $r$. See the discussion in the introduction of \cite{ALG} for more details. 

Suppose that $y\not = x$. In that case, we also take
$\mathcal{X}^y_0=\z_y$ (compare the last display with \eqref{formu-exit}). Then
under the probability measure $\N_x(\cdot\mid y\in[W_*,W^*])=\N_x(\cdot\mid \z_y>0)$, conditionally on $\z_y$,
the process $(\mathcal{X}^y_r)_{r\geq 0}$ is a 
continuous-state branching process with branching
mechanism $\varphi(u)=\sqrt{8/3}\,u^{3/2}$ (in short, a $\varphi$-CSBP)  started at $\z_y$. 
In particular, $(\mathcal{X}^y_r)_{r\geq 0}$ has a c\`adl\`ag modification, which we consider from now on.
We refer
to \cite[Chapter II]{Zurich}
for basic facts about continuous-state branching processes, and to \cite{ALG} for 
the preceding facts.

In the case $y=x$, we take $\mathcal{X}^x_0=0$ by convention, and the process $(\mathcal{X}^x_r)_{r\geq 0}$ is distributed under $\N_x$ according to the excursion measure 
of the $\varphi$-CSBP. This means that, if $\n=\sum_{k\in K}\delta_{\omega_k}$
is a Poisson point measure with intensity $\alpha\,\N_x$, the
process $X$ defined by $X_0=\alpha$ and, for every $r>0$,
$$X_r:=\sum_{k\in K} \mathcal{X}^x_r(\omega_k),$$
is a $\varphi$-CSBP started at $\alpha$ (see \cite[Section 2.4]{LGR}). 

In all cases, we call $(\mathcal{X}^y_r)_{r\geq 0}$ the exit measure process at $y$. 
Local times are related to this process by the formula
\begin{equation}
\label{local3}
\ell^y=\int_0^\infty \dd r \,\XX^y_r,
\end{equation}
which holds $\N_x$ a.e., for every $y\in\R$.
See \cite[Proposition 26]{LGR} when $y\not = x$, and \cite[Proposition 3.1]{LGR2} when $y=x$.

\subsection{The positive excursion measure}

We now introduce another important measure
on $\S_0$. There exists 
a $\sigma$-finite measure $\N^*_0$ on
$\S_0$, which is supported on the set $\S_0^+$ of nonnegative snake trajectories, such that,
for every bounded continuous function $G$ on $\S_0^+$ that vanishes on $\{\omega\in\S_0^+:W^*(\omega)\leq \delta\}$
for some $\delta >0$, we have
$$\N^*_0(G)=\lim_{\ve\to 0}\frac{1}{\ve}\,\N_\ve(G(\mathrm{tr}_{0}(\omega))).$$
See \cite[Theorem 23]{ALG}. Under $\N^*_0(\dd \omega)$, each  path $\omega_s$, for $0<s<\sigma$, starts from $0$, then
stays positive during some time interval $(0,u)$, and is stopped immediately
when it returns to $0$, if it does return to $0$.

Similarly to the definition of exit measures, one can make sense of the ``quantity'' of paths $\omega_s$ that return to $0$ under $\N^*_0$.
To this end, one proves that the limit
\begin{equation}
\label{approxz*0}
\z^*_0:=\lim_{\ve\to 0} \frac{1}{\ve^2} \int_0^\sigma \dd s\,\mathbf{1}_{\{\wh W_s<\ve\}}
\end{equation}
exists $\N^*_0$ a.e. See \cite[Section 10]{Disks}. 
Notice that replacing the 
limit by a liminf in \eqref{approxz*0} allows us to make sense of $\z^*_0(\omega)$ for every $\omega\in\S_0^+$.

We can then define conditional versions of the
measure $\N^*_0$, which play a fundamental role in the present work. Recall the definition of the
scaling operators $\theta_\lambda$ in \eqref{scalereln}. According to
\cite[Proposition 33]{ALG}, there exists a unique collection $(\N^{*,z}_0)_{z>0}$ of probability measures on $\S_0^+$
such that:
%\begin{enumerate}\label{N*props}
%\item[\rm(i)] We have
%$\;\N^*_0={\displaystyle
 %\sqrt{\frac{3}{2 \pi}} 
%\int_0^\infty \mathrm{d}z\,z^{-5/2}\, \N^{*,z}_0}$.
%\item[\rm(ii)] For every $z>0$, $\N^{*,z}_0$ is supported on $\{\z^*_0=z\}$.
%\item[\rm(iii)] For every $z,z'>0$, $\N^{*,z'}_0=\theta_{z'/z}(\N^{*,z}_0)$.
%\end{enumerate}
\begin{flalign}
\nonumber&\rm{(i)}\;\N^*_0={\displaystyle
 \sqrt{\frac{3}{2 \pi}} \int_0^\infty \mathrm{d}z\,z^{-5/2}\, \N^{*,z}_0}.&&\\
\label{N*props}&\text{(ii) For every $z>0$, $\N^{*,z}_0$ is supported on $\{\z^*_0=z\}$.}\\
&\nonumber\text{(iii) For every $z,z'>0$, $\N^{*,z'}_0=\theta_{z'/z}(\N^{*,z}_0)$.}
\end{flalign}
Informally, $\N^{*,z}_0=\N^*_0(\cdot\mid \z^*_0=z)$. Note that the ``law'' of $\z^*_0$ under $\N^*_0$ 
is the $\sigma$-finite measure
\begin{equation}
\label{Levy-mea}
\bn(\dd z)=\mathbf{1}_{\{z>0\}}\,\sqrt{\frac{3}{2 \pi}} 
\,z^{-5/2}\,\mathrm{d}z.
\end{equation}
It will be convenient to write $\check\N^{*,z}_0$
for the pushforward of $\N^{*,z}_0$ under the mapping $\omega\to -\omega$. Furthermore, for
every $a\in\R$, we write $\N^{*,z}_a$, resp. $\check\N^{*,z}_a$ for the pushforward of $\N^{*,z}_0$,
resp. of $\check\N^{*,z}_0$, under the shift $\omega\mapsto \omega + a$. 

We state a useful technical lemma.

\begin{lemma}
\label{maxN*}
For every $z>0$ and $\ve >0$, $\N^{*,z}_0(W^*<\ve)>0$. Moreover, there exists a constant $C$
such that, for every  $z>0$ and  $x>0$,
$$\N^{*,z}_0(W^*>x)\leq C \,\frac{z^3}{x^6}.$$
\end{lemma}

We omit the proof of the first assertion. For the second one, see \cite[Corollary 5]{LGR}. 
%
%The next theorem relates the measures $\N_x$ and $\N^*_0$ via a re-rooting transformation.
%Recall that, for every $\omega\in \S$ and every $s\in[0,\sigma(\omega)]$, 
%$\omega^{[s]}$ denotes the snake trajectory $\omega$ re-rooted at $s$ (Section \ref{sec:reroot}).
%
%\begin{theorem}{\rm\cite[Theorem 28]{ALG}}
%\label{re-root-theo}
%Let $G$ be a nonnegative measurable function on $\S$. Then,
%$$\N^*_0\Bigg(\int_0^\sigma \dd r\, G(\omega^{[r]})\Bigg)
%= 2\int_0^\infty \dd b\,\N_b\Big( \z_0\,G(\tr_0(\omega))\Big).
%$$
%\end{theorem}
%
\smallskip

Recall the notation $\mathcal{ Y}_{(\omega)}$ for the occupation measure
of $\omega\in\mathcal{S}$ from \eqref{occu-measure}.

\begin{lemma}
Let $z>0$. Then, $\N^{*,z}_0(\dd\omega)$ a.s. the measure $\mathcal{Y}_{(\omega)}$ has
a continuous density with respect to Lebesgue measure on $\R$. This density 
vanishes on $(-\infty,0]$ and  
is
continuously differentiable on $(0,\infty)$.
\end{lemma}

\proof Via scaling arguments, it is enough to prove this with $\N^{*,z}_0$ replaced by $\N^*_0$.
Then, we can use  the re-rooting property of $\N^{*}_0$ (see \cite[Theorem 28]{ALG}
or \cite[Theorem 5]{Mar})
to obtain that it suffices to prove the following claim: For every $b>0$, $\N_b(\dd \omega)$ a.e., the occupation
measure $\mathcal{Y}_{(\tr_0(\omega))}$ has a continuous density, which vanishes on $(-\infty,0]$ and  
is
continuously differentiable on $(0,\infty)$. Note that, $\N_b(\dd \omega)$ a.e., $\mathcal{Y}_{(\tr_0(\omega))}$ is supported on
$[0,\infty)$ and thus, once we know that $\mathcal{Y}(\tr_0(\omega))$  has
a continuous density it is obvious that this density vanishes on $(-\infty,0]$. 

Let us fix $b>0$ and argue under $\N_b$. Writing $(\omega^j)_{j\in J}$ for the excursions of $\omega$ away from $0$,
one easily verifies that, $\N_b(\dd \omega)$ a.e.,
\begin{equation}
\label{decompo}
\mathcal{Y}_{(\omega)}=\mathcal{Y}_{(\tr_0(\omega))}+ \sum_{j\in J} \mathcal{Y}_{(\omega^j)}.
\end{equation}
We know that $\N_b(\dd \omega)$ a.e., the measure $\mathcal{Y}_{(\omega)}$ has a 
continuously differentiable density $(\ell^x(\omega))_{x\in\R}$ and the same holds for 
the measures $\mathcal{Y}_{(\omega^j)}$ since we know that (conditionally on $\z_0(\omega)$)
the snake trajectories $\omega^j$, $j\in J$ are the atoms of a Poisson point measure with intensity $\z_0\N_0$.
Note that, for every fixed $x\not =0$, there are only finitely many indices $j$ such that $\ell^x(\omega^j)>0$. It
then follows that the measure $\sum_{j\in J} \mathcal{Y}_{(\omega^j)}$ has a density, and this density is given for $x\not =0$ by
the function
$\sum_{j\in J} \ell^x(\omega^j)$, which is continuously differentiable on $\R\backslash\{0\}$. However, $\N_b(\dd \omega)$ a.e., the function
$$x\mapsto \sum_{j\in J} \ell^x(\omega^j)$$
is continuous on $\R$: we already know that it is continuous on $\R\backslash\{0\}$, and for the 
continuity at $0$ we refer to formula (3.9) and the subsequent discussion in \cite{LGR2}. 
From \eqref{decompo}, we now deduce that $\mathcal{Y}_{(\tr_0(\omega))}$ has a continuous density on 
$\R$, which is given by 
$$x\mapsto \ell^x(\omega)-\sum_{j\in J} \ell^x(\omega^j).$$
This completes the proof. \endproof

In what follows, we will use the same notation $(\ell^x(\omega))_{x\in\R}$
to denote the density of $\mathcal{Y}_{(\omega)}$ under $\N^{*,z}_0(\dd\omega)$
or under $\N^{*,z}_a(\dd\omega)$ for any $a\in\R$. 

\subsection{Excursion theory}
\label{sec:excu}

Let us now recall the main theorem of the excursion theory developed 
in \cite{ALG}. We fix $x\in\R$ and $y\in[x,\infty)$, and we argue under $\N_x(\dd \omega)$.
As in the classical setting of excursion theory for linear
Brownian motion, our goal is to describe the evolution of the labels $V_u$ on the connected components 
of $\{u\in\t(\omega): V_u(\omega)\not =y\}$. 
So, let $\mathcal{C}$ be such a connected component and write $\ov{\mathcal{C}}$
for the closure of $\mathcal{C}$.
We leave aside the case where $\mathcal{C}$ contains the root $\rho_{(\omega)}$ of $\t(\omega)$
(this case does not occur if $y=x$).
Then, there
is a unique point $u$ of $\ov{\mathcal{C}}$ at minimal distance from $\rho_{(\omega)}$, such that
all points of $\ov{\mathcal{C}}$ are descendants of $u$,
and we have $V_u=y$.
Following \cite{ALG}, we say that $u$ is an excursion debut (from $y$). We can then
define a snake trajectory $\omega^{(u)}$ that accounts for the connected component $\mathcal{C}$ and the labels on $\mathcal{C}$.
To this end, we first observe that the set of all descendants of $u$ in $\t(\omega)$ can be written as $p_{(\omega)}([s_0,s'_0])$ ,
where $s_0$ and $s'_0$ are
such that $p_{(\omega)}(s_0)=p_{(\omega)}(s'_0)=u$. Then, we first define a snake trajectory $\tilde\omega^{(u)}\in \S_y$
coding the subtree $p_{(\omega)}([s_0,s'_0])$ (and its labels) by setting
$$\tilde\omega^{(u)}_s(t):=\omega_{(s_0+s)\wedge s'_0}(\zeta_{s_0}+t)\,\hbox{ for }0\leq t\leq \zeta_{(s_0+s)\wedge s'_0}-\zeta_{s_0}.$$
The set $\ov{\mathcal{C}}$ is the subset of $p_{(\omega)}([s_0,s'_0])$ consisting of all $v$ such that
labels stay greater than $y$ along the line segment from $u$ to  $v$, except at $u$ and possibly at $v$. This leads us to define
$$\omega^{(u)}:=\tr_y(\tilde\omega^{(u)}).$$
Then one can check (see \cite{ALG} for more details) that the compact $\R$-tree $\ov{\mathcal{C}}$ is identified 
isometrically to the tree $\t(\omega^{(u)})$, and moreover this identification
preserves labels. Also, the restriction of the volume measure of $\t(\omega)$ to $\mathcal{C}$
corresponds to the volume measure of $\t_{(\omega^{(u)})}$ via the latter identification.

We say that
$\omega^{(u)}$ is an excursion above $y$ if the values of $V_v$ for $v\in\mathcal{C}$ are greater than $y$
and  that $\omega^{(u)}$ is an excursion
below  $y$  if the values of $V_v$ for $v\in\mathcal{C}$ are smaller than $y$. Note that an excursion away from $y$,
as considered in Proposition \ref{SMP}, will contain infinitely many
excursions above or below $y$. Let $\mathcal{Y}_{(\omega)}^{(y,\infty)}$ denote the
restriction of $\mathcal{Y}_{(\omega)}$ to $(y,\infty)$. Then, the preceding identification of
volume measures entails that
\begin{equation}
\label{decomp-occup}
\mathcal{Y}_{(\omega)}^{(y,\infty)}=\sum_{u\in\mathcal{D}_y^+} \mathcal{Y}_{(\omega^{(u)})},
\end{equation}
where $\mathcal{D}_y^+$ is the set of all debuts of excursions above $y$.

Recall that the exit measure process $(\mathcal{X}^y_r)_{r\geq 0}$ was defined in Section \ref{sna-mea}.
By Proposition 3 of \cite{ALG} (and an application of the special Markov property when $y\not=x$), excursion debuts 
from $y$ are in one-to-one 
correspondence with the jump times of the process $(\mathcal{X}^y_r)_{r\geq 0}$,
or equivalently with the jumps of this process,
in such a way that, if $u$ is an excursion debut and $s\in[0,\sigma]$ is such that
$p_{(\omega)}(s)=u$, the associated jump time of the exit measure process at $y$
is the total local time at $y$ accumulated by the path $W_s$. We can 
rank the jumps of $(\mathcal{X}^y_r)_{r\geq 0}$ in a sequence $(\delta_i)_{i\in\N}$
in decreasing order.
For every $i\in\N$, we write $u_i$ for the excursion debut associated with the jump $\delta_i$.  
The following theorem is essentially Theorem 4 in \cite{ALG}. We write $\N_x^{(y)}=\N_x(\cdot \mid \z_y>0)$ when $y\not=x$,
and $\N_x^{(x)}=\N_x$.

\begin{theorem}
\label{theo-excursion}
Under $\N_x^{(y)}$, conditionally on $(\mathcal{X}^y_r)_{r\geq 0}$,
the excursions $ \omega^{(u_i)}$, $i\in \N$, are independent, and independent
of $\tr_y(\omega)$, and, for every $i\in\N$, the 
conditional distribution of $\omega^{(u_i)}$ is 
$$\frac{1}{2}\Big(\N^{*,\delta_i}_y+\check\N^{*,\delta_i}_y\Big).$$
\end{theorem}
\noindent We say that $\delta_i$ is the boundary size of the excursion $\omega^{(u_i)}$.

The case $y=x$ of Theorem \ref{theo-excursion} is Theorem 4 of \cite{ALG} 
and the case $y\not =x$ can then be derived by an application of the special Markov property (Proposition \ref{SMP}).

\subsection{The L\'evy bridge}
\label{subsec:bridge}

Recall from the Introduction and Section~\ref{sna-mea} that for  $\lambda\geq 0$, $\psi(\lambda)=\frac{1}{2}\varphi(\lambda)=\sqrt{2/3}\,\lambda^{3/2}$, and that $(U_t)_{t\geq 0}$ denotes a
 stable L\'evy process with index $3/2$, without negative jumps, and scaled so that its Laplace exponent is $\psi(\lambda)$. This means that
for every $t\geq 0$ and $\lambda >0$, we 
have
$$\E[\exp(-\lambda U_t)]=\exp(t\psi(\lambda)).$$
The L\'evy measure of $U$ is $\frac{1}{2}\bn(\dd z)$, where $\bn(\dd z)$ was defined in \eqref{Levy-mea}, and
$U_s$ has characteristic function
$$\E[e^{\II u U_s}]= e^{-s\Psi(u)},$$
where
\begin{equation}
\label{carac-U}
\Psi(u)= c_0 |u|^{3/2}\,(1+\II\,\mathrm{sgn}(u)),
\end{equation}
and $c_0=1/\sqrt{3}$. Recall also that $U_s$ has a density, $p_s(x)$, which by Fourier inversion is given by
 $$p_s(x)=\frac{1}{2\pi}\int e^{-\II ux-s\Psi(u)}\dd u.$$
Several properties of $p_s(x)$ were recalled in the Introduction. Another property we use is that  the distribution 
of $U_s$ is known to be unimodal, in the sense that there exists $a\in\R$ such that both functions
$x\mapsto p_s(a-x)$ and $x\mapsto p_s(a+x)$ are nonincreasing on $\R_+$
(cf. \cite[Theorem 2.7.5]{Zol}). 

For every $t>0$ and $y\in\R$, we can make sense of the process $(U_s)_{0\leq s\leq t}$ conditioned on $\{U_t=y\}$,
which is called the
$\psi$-L\'evy bridge from $0$ to $y$ in time $t$ (see \cite{FPY} for a construction in a much more general setting). 
Write $(U^{\mathrm{br},t,y})_{0\leq s\leq t}$ for a $\psi$-L\'evy bridge from $0$ to $y$ in time $t$. Then,
for every $r\in(0,t)$ and every nonnegative measurable function $F$ on the Skorokhod space
$\D([0,r],\R)$, we have
\begin{equation}\label{RNF}\E\Big[F\Big((U^{\mathrm{br},t,y})_{0\leq s\leq r}\Big)\Big] =\E\Bigg[ \frac{p_{t-r}(y-U_r)}{p_t(y)}\,F\Big((U_s)_{0\leq s\leq r}\Big)\Bigg].
\end{equation}
See \cite[Proposition 1]{FPY}. In particular, the law of $(U^{\mathrm{br},t,y}_s)_{0\leq s\leq r}$ has a bounded density 
with respect to the law of $(U_s)_{0\leq s\leq r}$. Via a simple time-reversal argument, the same holds for the
law of $(y-U^{\mathrm{br},t,y}_{(t-s)-})_{0\leq s\leq r}$. 

In what follows, when we write 
$$\E\Big[F\Big((U_s)_{0\leq s\leq t}\Big)\,\Big|\,U_t=y\Big],$$
this should always be understood as $\E[F( (U^{\mathrm{br},t,y})_{0\leq s\leq t})]$ (which makes sense 
for every choice of $y\in\R$).

\section{The connection with super-Brownian motion}
\label{sec:super}

Let us briefly recall the connection between the Brownian snake excursion measures $\N_x$
and super-Brownian motion, referring to \cite{Zurich} for more details. We fix $\alpha>0$, and consider a Poisson point measure on $\S$,
$$\n=\sum_{k\in K} \delta_{\omega_k}$$
with intensity $\alpha\,\N_0$. Then one can construct a 
one-dimensional super-Brownian motion $(\mathbf{X}_t)_{t\geq 0}$ 
with branching mechanism $\Phi(u)=2u^2$ and initial value $\mathbf{X}_0=\alpha\,\delta_0$, such that,
for any nonnegative measurable function $f$ on $\R$,
\begin{equation}
\label{occu-super}
\int_0^\infty \mathbf{X}_t(f)\, \dd t= \sum_{k\in K} \mathcal{Y}_{(\omega_k)}(f)
\end{equation}
where $\mathcal{Y}_{(\omega_k)}$ is defined in formula \eqref{occu-measure}. 
In a more precise way, the process $(\mathbf{X}_t)_{t\geq0}$ is defined by setting, for 
every $t>0$ and every nonnegative Borel function $f$ on $\R$,
$$\mathbf{X}_t(f)  := \sum_{k\in K}  \int_0^{\sigma(\omega_k)} f(\wh W_r(\omega_k))\,\dd_rl^t_r(\omega_k),$$
where $l^t_r(\omega_k)$ denotes the local time of the process $s\mapsto \zeta_s(\omega_k)$ at level $t$ and at time $r$, 
and the notation $\dd_rl^t_r(\omega_k)$ refers to integration with respect to the nondecreasing function $r\mapsto l^t_r(\omega_k)$
 (see Chapter 4 of \cite{Zurich}).
 
 The preceding representation of $\mathbf{X}$ allows us to consider excursions above and below $a$, for any $a\in\R$. 
 Consider for simplicity the case $a=0$. We define the exit measure process $(X^0_t)_{t\geq 0}$ at $0$
by setting $X^0_0=\alpha$ and, for $t>0$,
\begin{equation}\label{exitmdef}X^0_t=\sum_{k\in K} \mathcal{X}^0_t(\omega_k).
\end{equation}
As was already mentioned in Section \ref{sna-mea}, the process $(X^0_t)_{t\geq 0}$ is a $\varphi$-CSBP started at $\alpha$. Write $(\delta_i)_{i\in\N}$
for the sequence of its jumps ordered in decreasing size.
Then the collection of all excursions of $\omega_k$ above and below $0$, combined for all $k\in K$, is in one-to-one correspondence with 
the collection $(\delta_i)_{i\in\N}$. Moreover, if $\omega_i$ denotes the excursion associated with the jump $\delta_i$, then:
\begin{align}\label{indtexc}
&\text{The excursions 
$\omega_i$, $i\in \N$, are independent conditionally on $(X^0_t)_{t\geq 0}$,} \\
\nonumber&\text{and the conditional distribution of 
$\omega_i$ is
$\frac{1}{2}\Big(\N^{*,\delta_i}_y+\check\N^{*,\delta_i}_y\Big)$.}
\end{align}
All these facts are immediate consequences of Theorem \ref{theo-excursion} and the discussion preceding it. 
  
We are primarily interested in the total occupation measure
$$\mathbf{Y}:=\int_0^\infty \mathbf{X}_t \,\dd t.$$
Recall from the Introduction, the notation $L^x$, $\dot L^x$ for its continuous density, and its continuous derivative on 
$\{x\neq 0\}$, and $\dot L^{0+}$, $\dot L^{0-}$ for the right and left derivatives at $0$, respectively, and $\dot L^0:=\dot L^{0+}$.  It also follows from 
Sugitani \cite[Theorem 4]{Sug} and its proof that 
$$\dot L^{0+}=\lim_{x\to 0, x>0} \dot L^x\,,\quad \dot L^{0-}=\lim_{x\to 0, x<0} \dot L^x\,,$$
and
\begin{equation}\label{L0jump}\dot L^{0+}-\dot L^{0-}= -2\alpha.
\end{equation}

Fix $a\geq 0$, and write $\mathbf{Y}^{(a,\infty)}$ for the restriction of $\mathbf{Y}$
to $(a,\infty)$, and similarly $\mathcal{Y}_{(\omega_k)}^{(a,\infty)}$ for the restriction of $\mathcal{Y}_{(\omega_k)}$ to $(a,\infty)$. In what follows, we assume that 
$\{k\in K:W^*(\omega_k)>a\}$ is not empty. In view of our applications, we are interested in
excursions of $\omega_k$ above level $a$, combined for all $k\in K$, such that $W^*(\omega_k)>a$. We can order these excursions in a 
sequence $(\omega^{a,+}_j)_{j\in\N}$ in decreasing order of their boundary sizes (Theorem \ref{theo-excursion} implies that these
boundary sizes are distinct a.s.). From \eqref{decomp-occup} and \eqref{occu-super}, we have
$$\mathbf{Y}^{(a,\infty)}=\sum_{k\in K} \mathcal{Y}_{(\omega_k)}^{(a,\infty)}=\sum_{j\in \N} \mathcal{Y}_{(\omega^{a,+}_j)}.$$
Consequently, for every $h>0$, we have
\begin{equation}
\label{decomp-local}
L^{a+h}=\sum_{j=1}^\infty \ell^{a+h}(\omega^{a,+}_j)\,,\quad \dot L^{a+h}=\sum_{j=1}^\infty \dot\ell^{a+h}(\omega^{a,+}_j)\,.
\end{equation}
Note that there are only finitely many nonzero terms in the sums of the last display.

The next 
proposition will be a key ingredient of our approach. We can write the supremum of the support of $\mathbf{Y}$ as $R=\sup\{W^*(\omega_k):k\in K\}$.
 By \eqref{Lsupp}, we have for any $a\ge 0$, 
$\{L^a>0\}=\{R>a\}$, a.s.

\begin{proposition}
\label{key-ingre}
Let $a\geq 0$, let $F$ be a nonnegative measurable function on the space $C((-\infty,a],\R_+\times \R)$, 
and let $G$ be a nonnegative measurable function on 
$(\mathcal{S}_a)^\N$. Then,
$$\E\Big[\mathbf{1}_{\{R>a\}}\,F\Big((L^x,\dot L^x)_{x\in(-\infty,a]}\Big) G\Big((\omega^{a,+}_j)_{j\in\N}\Big)\Big]
=\E\Big[ F\Big((L^x,\dot L^x)_{x\in(-\infty,a]}\Big) \Phi_G(L^a,\frac{1}{2}\dot L^a)\Big]$$
where $\Phi_G(0,y)=0$ for every $y\in\R$, and, for every $t>0$ and $y\in\R$, $\Phi_G(t,y)$ is defined as follows. Let
$U^{\mathrm{br},t,y}$ be a $\psi$-L\'evy bridge from $0$ to $y$ in time $t$, and let
$(Z_j)_{j\in\N}$ be the collection of jumps of $U^{\mathrm{br},t,y}$ ordered in
nonincreasing size. Then,
$$\Phi_G(t,y)=\E\Big[G\Big((\varpi_j)_{j\in\N}\Big)\Big],$$
where, conditionally on $(Z_j)_{j\in\N}$, the random snake trajectories
$(\varpi_j)_{j\in\N}$ are independent, and, for every $j$, $\varpi_j$
is distributed according to $\N^{*,Z_j}_a$.
\end{proposition}

See \cite[Section 6]{Mar} for a proof of this proposition (cf.
formula (38) in \cite{Mar}). Proposition \ref{key-ingre} is basically a consequence 
of Theorem \ref{theo-excursion}, but one needs to understand the conditional distribution
of the boundary sizes of excursions above level $a$ given the collection of boundary sizes of
excursions below $a$, see in particular formula (24) in \cite{Mar}.

 Thanks to formula \eqref{decomp-local}, Proposition \ref{key-ingre} immediately gives
 the (time-homogeneous) Markov property of the process $(L^x,\dot L^x)_{x\geq 0}$. Moreover, this
proposition shows that, for every $t>0$ and $y\in\R$, the conditional distribution of $(\omega^{a,+}_j)_{j\in\N}$
knowing $L^a=t$ and $\frac{1}{2}\dot L^a=y$ is the law of the sequence $(\varpi_j)_{j\in\N}$,
as described in the statement. We emphasize that
this conditional distribution makes sense for {\it every} choice of $t>0$ and $y\in\R$. Later, when we
consider expressions of the form 
\begin{equation}
\label{condi-expec}
\E\Big[G\Big((\omega^{a,+}_j)_{j\in\N}\Big)\,\Big|\,L^a=t, \frac{1}{2}\dot L^a=y\Big],
\end{equation}
this will always mean that we integrate $G$ with respect to the conditional distribution described above. 

\section{Moment Bounds and Quadratic Variation}
\label{est-moment}

In this section, we use a representation due to Hong \cite{Hon} to derive certain estimates for
moments of the derivatives $\dot L^x$ introduced in the previous section. We consider the
super-Brownian motion $\mathbf{X}$ with $\mathbf{X}_0=\alpha\delta_0$, constructed as above, and write $M$
for the associated martingale measure (see \cite[Section II.5]{Per}. For every function
$\phi:\R\la \R$ of class $C^2$, 
$$M_t(\phi):=\mathbf{X}_t(\phi) - \mathbf{X}_0(\phi) -\int_0^t \mathbf{X}_s(\phi''/2)\,\dd s$$
is a (continuous) local martingale (with respect to the canonical filtration of $\mathbf{X}$) with quadratic variation
\begin{equation}
\label{quad-var-M}
\langle M(\phi),M(\phi)\rangle_t= 4\int_0^t \mathbf{X}_s(\phi^2)\,\dd s.
\end{equation}
There is a linear extension of the definition of the local martingale $M_t(\phi)$ to locally bounded Borel functions $\phi$ and 
\eqref{quad-var-M} remains valid (e.g., see Proposition~II.5.4 and Corollary~III.1.7 of \cite{Per}).

Let $\xi:=\inf\{t\geq 0:\mathbf{X}_t=0\}$ stand for the (a.s. finite) extinction time of $\mathbf{X}$ and let $x>0$. According to \cite[Proposition 2.2]{Hon}, we have a.s. 
for every $t\geq \xi$,
\begin{equation}
\label{repre-deri}
\dot L^x=-\alpha-M_t(\mathrm{sgn}(x-\cdot)),
\end{equation}
where $\mathrm{sgn}(x-\cdot)$ stands for the function $y\mapsto \mathbf{1}_{\{x>y\}}-\mathbf{1}_{\{x<y\}}$. With our convention 
for $\dot L^0$, this formula remains valid for $x=0$. 
We use this representation to derive the following lemma.

\begin{lemma}
\label{moment-deri}
{\rm (i)} For every $q\in [1,4/3)$, for every $x,y\in \R$,
$$\E[|\dot L^x-\dot L^y|^{q}]<\infty.$$
{\rm (ii)} Let $q\in [1,4/3)$. There exists a constant $\beta>0$ such that, for every $0<u<v$,
\begin{equation}
\label{bound-incre}
\E\Big[\sup_{x, y\in[u,v],x\neq y}\Big( \frac{ |\dot L^x -\dot L^y|}{|x-y|^\beta}\Big)^q\Big]<\infty.
\end{equation}
\end{lemma}

\proof (i) We first verify that, for every $x>0$ and every $q\in(0,2/3)$,
\begin{equation}
\label{deri-tec1}
\E\Big[\Big(\int_0^\infty \mathbf{X}_s([0,x])\,\dd s\Big)^q\Big]<\infty.
\end{equation}
To see this, recall the well-known formula $\P(\xi>t)=1-\exp(-\frac{\alpha}{2t})$ (which is easily derived
from the representation of the preceding section), and write for every $\lambda>0$ and $r>0$,
\begin{align*}
\P\Bigg(\Big(\int_0^\infty \mathbf{X}_s([0,x])\,\dd s\Big)^q>\lambda\Bigg)
&\leq \P(\xi >\lambda^r)+\P\Big(\int_0^{\lambda^r} \mathbf{X}_s([0,x])\,\dd s>\lambda^{1/q}\Big)\\
&\leq \frac{\alpha}{2\lambda^r} + \frac{1}{\lambda^{1/q}} \int_0^{\lambda^r} \E[\mathbf{X}_s([0,x])]\,\dd s\\
&= \frac{\alpha}{2\lambda^r}+ \frac{\alpha}{\lambda^{1/q}} \int_0^{\lambda^r} \P(B_s\in[0,x])\,\dd s\\
&\leq \alpha\Big(\frac{1}{2}\,\lambda^{-r}+ x\,\lambda^{r/2-1/q}\Big),
\end{align*}
where we wrote $(B_t)_{t\geq 0}$ for a linear Brownian motion started at $0$, and  we
used the trivial bound $\P(B_s\in[0,x])\leq x/\sqrt{2\pi s}$. If we take $r=2/(3q)$, the 
right-hand side of the previous display becomes a constant, depending on $x$, times $\lambda^{-2/(3q)}$, which
is integrable in $\lambda$ with respect to Lebesgue measure on $[1,\infty)$ if $0<q<2/3$. Our claim 
\eqref{deri-tec1} follows. 

Next let $K>0$ and $0\leq x<y\leq K$. We observe that $M_t(\mathrm{sgn}(x-\cdot))-M_t(\mathrm{sgn}(y-\cdot))$
is a continuous local martingale with quadratic variation
$$4\int_0^t \mathbf{X}_s((\mathrm{sgn}(x-\cdot)-\mathrm{sgn}(y-\cdot))^2)\,\dd s= 16\int_0^t\mathbf{X}_s([x,y])\,\dd s.$$
From \eqref{deri-tec1} and the Burkholder-Davis-Gundy inequalities, we obtain that, for every $q\in[1,4/3)$,
$$\E\Big[ \Big|M_t(\mathrm{sgn}(x-\cdot))-M_t(\mathrm{sgn}(y-\cdot))\Big|^q\Big]\leq C_{(q,K)},$$
where the constant $C_{(q,K)}$ only depends on $K$ and $q$. Letting $t$ tend to infinity and 
using \eqref{repre-deri} together with Fatou's lemma, we get that $\E[|\dot L^x-\dot L^y|^{q}]\leq C_{(q,K)}$.
By symmetry, we have
for every $x>0$, $\E[|\dot L^{-x}-\dot L^{0-}|^{q}]=\E[|\dot L^x-\dot L^0|^{q}]<\infty$, and, by \eqref{L0jump}, $|L^0-L^{0-}|=2\alpha$. Assertion (i) follows.

\smallskip
\noindent (ii) We first observe that, for every $\delta>0$, there is a constant $C_\delta$ (depending on $\alpha$)  such that,
for every $\delta\leq x\leq y$ and every $s>0$,
\begin{equation}\label{L2int}\E[\mathbf{X}_s([x,y])^2]\leq C_\delta\,(y-x)^2.
\end{equation}
To see this first use the explicit formula 
$$\E[\mathbf{X}_s([x,y])^2]= \alpha^2\Bigg(\int_x^y q_s(u)\dd u \Bigg)^2 + 4\alpha\int_0^s \dd r\int_\R \dd u \,q_r(u)\Bigg(\int_x^y \dd v\, q_{s-r}(v-u)\Bigg)^2,$$
where $q_s(u)$ is the Brownian transition density (see e.g. Proposition II.11 in \cite{Zurich}). To handle 
 the second term of the right-hand side, bound $q_{s-r}(v-u)$ by $C/\sqrt{s}$ when $r<s/2$, and  when $r>s/2$
use $\int \dd u \,q_{s-r}(v-u)q_{s-r}(v'-u)=q_{2(s-r)}(v-v')$. The bound \eqref{L2int} now follows from a short calculation.

To simplify notation, set $\wh L^x_t=-\alpha-M_t(\mathrm{sgn}(x-\cdot))$. From the Burkholder-Davis-Gundy inequalities and the bound in \eqref{L2int}, we get the existence
of a constant $C$ such that, for every $\delta\leq x\leq y$,
$$\E[(\wh L^y_t-\wh L^x_t)^4]\leq C\,\E\Big[\Big(\int_0^t \mathbf{X}_s([x,y])\,\dd s\Big)^2\Big] \leq C\,t\,\E\Big[\int_0^t (\mathbf{X}_s([x,y]))^2\,\dd s\Big] \leq C\,C_\delta\,t^2(y-x)^2.$$
Let $a>0$ and $\lambda>0$. For every $n\in\N$, we can bound
$$\P\Bigg(\sup_{1\leq k\leq 2^n} |\wh L^{1+k2^{-n}}_t- \wh L^{1+(k-1)2^{-n}}_t| >\lambda a^n\Bigg)
\leq 2^n\times (\lambda a^n)^{-4} \times C\,C_1\,t^22^{-2n}=C\,C_1\,t^2\lambda^{-4}a^{-4n}2^{-n}.$$
We fix $a\in(0,1)$ such that $a^{-4}<2$.
Consider the event 
$$A:=\bigcup_{n\in\N} \Bigg\{\sup_{1\leq k\leq 2^n} |\wh L^{1+k2^{-n}}_t- \wh L^{1+(k-1)2^{-n}}_t| >\lambda a^n\Bigg\}.$$
We get 
$\P(A)\leq \wt C\,t^2\lambda^{-4}$,
where $\wt C$ is a constant. Let $D$ be the set of all real numbers of the form $1+k2^{-n}$ with $n\in\N$ and $k\in\{0,1,\ldots, 2^n\}$ On the complement of the set $A$, simple chaining arguments show that we have $|\wh L^x_t-\wh L^y_t|\leq K\,\lambda\,|x-y|^\beta$ for every $x,y\in D$, where $\beta=-\log a/\log 2>0$ and $K$
is a constant (which does not depend on $\lambda$). Finally, since $\dot L^y -\dot L^x=\wh L^y_t-\wh L^x_t$ on $\{\xi\leq t\}$, we have 
$$\P\Bigg(\sup_{x,y\in[1,2],x\neq y} \frac{ |\dot L^x -\dot L^y|}{|x-y|^\beta}>K\,\lambda\Bigg)
=\P\Bigg(\sup_{x,y\in D,x\neq y} \frac{ |\dot L^x -\dot L^y|}{|x-y|^\beta} >K\,\lambda\Bigg)\leq \P(\xi>t) + \wt C\,t^2\lambda^{-4}\leq\frac{\alpha}{2t} + \wt C\,t^2\lambda^{-4}.$$
We apply this bound with $t=\lambda^{4/3}$, and it follows that 
$$\E\Big[\Big(\sup_{x,y\in[1,2],x\neq y}\Big( \frac{ |\dot L^x -\dot L^y|}{|x-y|^\beta}\Big)^q\Big]<\infty
$$
for every $q\in[1,4/3)$. By a minor modification of the argument, the last display still holds if we
replace $[1,2]$ by any interval $[u,v]$ with $0<u<v$. \endproof

The following proposition determines 
the quadratic variation of $(\dot L^x)_{x\geq 0}$. We will see later that this process is a semimartingale (for an appropriate filtration). 

\begin{proposition}
\label{quad-var}
Let $\ov x>0$, and, for every integer $n\in\N$, let $\pi_n=\{0=x^n_0<x^n_1<\cdots< x^n_{m_n}=\ov x\}$
be a subdivision of $[0,\ov x]$. Set $\|\pi_n\|:=\max\{x^n_i-x^n_{i-1}:1\leq i\leq m_n\}$, and 
$$Q(\pi_n)= \sum_{i=1}^{m_n} (\dot L^{x^n_i}-\dot L^{x^n_{i-1}})^2.$$
Assume that $\|\pi_n\| \la 0$ as $n\to\infty$.
Then,
$$Q(\pi_n)\build{\la}_{n\to\infty}^{} 16 \int_0^{\ov x} L^x\,\dd x\quad\hbox{in probability}.$$
\end{proposition}

\proof
We use the same notation 
$\wh L^x_t=-\alpha-M_t(\mathrm{sgn}(x-\cdot))$, for $x\geq 0$ and $t\geq 0$, 
as in the previous proof, and we recall that
$\dot L^x=\wh L^x_t$ when $t\geq \xi$, by \eqref{repre-deri}. If $0\leq x\leq y$, we have
$$\wh L^y_t -\wh L^x_t= -2\,M_t(\mathbf{1}_{[x,y]}).$$
Fix a subdivision $\pi=\{0=x_0<x_1<\cdots< x_{m}=\ov x\}$ of $[0,\ov x]$. We will use the
last display to evaluate 
$$Q_t(\pi):=\sum_{i=1}^m (\wh L^{x_i}_t-\wh L^{x_{i-1}}_t)^2.$$
For every $i\in\{1,\ldots,m\}$, set
$$M^i_t:=-2\,M_t(\mathbf{1}_{[x_{i-1},x_i]})$$
so that $M^i$ is a local martingale with quadratic variation
$$\langle M^i,M^i\rangle_t= 16 \int_0^t\,\mathbf{X}_s([x_{i-1},x_i])\,\dd s.$$
Also set
$$N^i_t:=(M^i_t)^2-\langle M^i,M^i\rangle_t = 2\int_0^t M^i_s\,\dd M^i_s.$$
Then,
 \begin{align}
 \label{QV-tech1}
\E\Big[\Big(Q_t(\pi)-16 \int_0^t\,\mathbf{X}_s([0,\ov x])\,\dd s\Big)^2\Big]
& = \E\Big[\Big(\sum_{i=1}^m \big((M^i_t)^2 - \langle M^i,M^i\rangle_t\big)\Big)^2\Big]\nonumber\\
 &= \E\Big[\sum_{i=1}^m (N^i_t)^2\Big] + 2\sum_{1\leq i<j\leq m} \E[N^i_tN^j_t].
 \end{align}
 On one hand, we have $\E[N^i_tN^j_t]=0$ if $i\not =j$, because 
 $$\langle M^i,M^j\rangle_t=16\int_0^t \mathbf{X}_s([x_{i-1},x_i]\cap[x_{j-1},x_j])\,\dd s = 0$$
 and $N^i_t$ is a stochastic integral with respect to $M^i$ Note that integrability issues are
 trivial here because the random variables $\mathbf{X}_s(\R)$, $0\leq s\leq t$, are uniformly
 bounded in $L^p$, for any $p<\infty$ (e.g., see Lemma III.3.6 of \cite{Per}). On the other hand, we
 can estimate $\E[(N^i_t)^2]$ as follows. Using the Burkholder-Davis-Gundy inequalities and
 writing $C_1$  and $C_2$ for the appropriate constants, we have
 \begin{align*}
 \E[(N^i_t)^2]&\leq 2\Big( \E[(M^i_t)^4]+\E[(\langle M^i,M^i\rangle_t )^2]\Big)\\
 &\leq C_1\,\E[(\langle M^i,M^i\rangle_t )^2]\\
 &=C_2 \,\E\Big[\int_0^t \dd s\int_s^t \dd r\,\mathbf{X}_s([x_{i-1},x_i])\mathbf{X}_r([x_{i-1},x_i])\Big]\\
 &=C_2\,\int_0^t \dd s\int_s^t \dd r\,\E\Big[\mathbf{X}_s([x_{i-1},x_i])\,\E_{\mathbf{X}_s}[\mathbf{X}_{r-s}([x_{i-1},x_i])]\Big]\\
 &\leq C_2\,\int_0^t \dd s\int_s^t \dd r\,\frac{x_i-x_{i-1}}{2\sqrt{r-s}}\, \E\Big[\mathbf{X}_s([x_{i-1},x_i])\mathbf{X}_s(\R)\Big]\\
 &\leq C_2\,(x_i-x_{i-1})\,\sqrt{t}  \int_0^t \dd s \,\E\Big[\mathbf{X}_s([x_{i-1},x_i])\mathbf{X}_s(\R)\Big].
 \end{align*}
 In the fourth line of this calculation, we applied the Markov property of $\mathbf{X}$, writing $\P_\mu$
 for a probability measure under which $\mathbf{X}$ starts from $\mu$, and, in the next line, we used the
 first-moment formula for $\mathbf{X}$.
 By summing the estimate of the last display over $i\in\{1,\ldots,m\}$, we get
 $$\E\Big[\sum_{i=1}^m (N^i_t)^2\Big]\leq C_2\,\|\pi\| \sqrt{t}\,\int_0^t \E[\mathbf{X}_s(\R)^2]\,\dd s\leq C_2\,\|\pi\| \sqrt{t}\,(\alpha^2t+2\alpha t^2),$$
 using the simple estimate $\E[\mathbf{X}_s(\R)^2]\leq \alpha^2 +4\alpha s$. Finally, we deduce from \eqref{QV-tech1} that, for $t\geq 1$,
 $$\E\Big[\Big(Q_t(\pi)-16 \int_0^t\,\mathbf{X}_s([0,\ov x])\,\dd s\Big)^2\Big]\leq C_3\,t^{5/2}\,\|\pi\|.$$
 We apply the latter estimate to $\pi=\pi_n$ for every $n\geq 1$, and it follows that, for every $t\geq 1$,
 $$\lim_{n\to\infty}\E\Big[\Big(Q_t(\pi_n)-16 \int_0^t\,\mathbf{X}_s([0,\ov x])\,\dd s\Big)^2\Big]=0.$$
 Since
 $$\P\Big(Q_t(\pi_n)=Q(\pi_n),\int_0^t\,\mathbf{X}_s([0,\ov x])\,\dd s=\int_0^\infty\mathbf{X}_s([0,\ov x])\,\dd s\Big) \geq \P(\xi\leq t) \build\la_{t\to\infty}^{} 1,$$
 this immediately gives the convergence in probability
 $$Q(\pi_n)\build\la_{n\to\infty}^{} 16 \int_0^\infty\,\mathbf{X}_s([0,\ov x])\,\dd s = 16\int_0^{\ov x} L^x\,\dd x.\eqno{\square}$$

\section{The expected value of increments of the derivative of local time}
\label{sec:exp-incre-deriv}

\subsection{The case of the positive excursion measure}
\label{sec:casepositive}
	
Our goal in this section is to compute the quantities $\N^{*,z}(\dot\ell^a)$
for $z>0$ and $a>0$. We start with a technical estimate.

\begin{lemma}
\label{lem-tech}
Let $q\in[1,4/3)$. Then, for
every $0<u<v$, and $n\in\N$,
$$ \sup_{1/n\le z\le n}\N^{*,z}_0\Big(\Big(\sup_{u\leq x\leq v}|\dot \ell^x|\Big)^q\Big) <\infty.$$
\end{lemma}

\proof We will derive this result from Lemma \ref{moment-deri}, using the 
construction of  the super-Brownian motion $(\mathbf{X}_t)_{t\geq 0}$
 in Section \ref{sec:super}. Recall the definition 
of the exit measure process $(X^0_t)_{t\geq 0}$ in \eqref{exitmdef} and that it is a $\varphi$-CSBP, where $\varphi=2\psi$.
By the Lamperti transformation \cite{Lam}, we can write $X^0$ as a (continuous) time change of a L\'evy process with no negative jumps and Laplace exponent $\varphi$, started at $\alpha$, up to its first hitting time of $0$. Up to enlarging the probability space, we may assume that this L\'evy process 
$(\mathcal{U}_t)_{t\geq 0}$ is defined for all $t\geq 0$ and we write $T_0=\inf\{t\geq 0:\mathcal{U}_t=0\}$. Notice that the jumps 
of $X^0$ are exactly the jumps of $\mathcal{U}$ on the time interval $[0,T_0]$.

Let us fix $0<u<v$.
Let $b>0$, and let $\mathcal{U}^{(1)}$ be the L\'evy process that only records the jumps of $\mathcal{U}$ of size greater than $b$,
$$\mathcal{U}^{(1)}_t := \sum_{s\leq t} \Delta \mathcal{U}_s\,\mathbf{1}_{\{\Delta \mathcal{U}_s >b\}}.$$
Also set $\mathcal{U}^{(0)}_t:=\mathcal{U}_t-\mathcal{U}^{(1)}_t$, so that $\mathcal{U}^{(0)}$ and $\mathcal{U}^{(1)}$ are two independent L\'evy 
processes, with $\mathcal{U}^{(1)}_0=0$ and $\mathcal{U}^{(0)}_0=\alpha$. We can find a constant $t_1>0$ such that
the probability of the event $A$ where $\mathcal{U}^{(1)}$ has exactly one jump during $[0,t_1]$ and $\mathcal{U}^{(0)}$ does not hit $0$ before $t_1$ is positive. On the event $A$,
let $\Delta_0$ be the unique jump of $\mathcal{U}^{(1)}$ on the time interval $[0,t_1]$. Then, conditionally on the event $A$,
$\Delta_0$ is distributed according to the probability measure $(3b^{3/2}/2)\,\mathbf{1}_{(b,\infty)}(z)z^{-5/2}\,\dd z$.
On the event $A$, let $\omega_0$ be the excursion of $\mathbf{X}$ (above or below $0$)  associated with the jump $\Delta_0$.  Here, recall the definition of these excursions
in Section \ref{sec:super}, and the fact that they are in one-to-one correspondence with the jumps of $X^0$, or equivalently
the jumps of $\mathcal{U}$ on $[0,T_0]$ (see especially \eqref{indtexc} and the discussion prior to it). Also let $A'$ be the event where
all excursions of $\mathbf{X}$ above or below $0$, except possibly the excursion $\omega_0$ (if it is defined), stay in the interval $(-1,u)$. On the event $B=A
\cap A'$, we have $L^a=\ell^a(\omega_0)$ for every $a\notin (-1,u)$. Then, on one hand, it follows from Lemma \ref{moment-deri} that
\begin{equation}
\label{estim1}
\E\Big[\mathbf{1}_{B}\,\Big(\sup_{x\in[u,v]} |\dot L^x-\dot L^{-1}|\Big)^q\Big]<\infty.
\end{equation}
On the other hand, the preceding remarks give
\begin{align}
\label{estim2}
&\E\Big[\mathbf{1}_{B}\,\Big(\sup_{x\in[u,v]} |\dot L^x-\dot L^{-1}|\Big)^q\Big]\nonumber\\
&=\E\Big[\mathbf{1}_{B}\,\Big(\sup_{x\in[u,v]} |\dot\ell^x(\omega_0)-\dot\ell^{-1}(\omega_0)|\Big)^q\Big]\nonumber\\
&=\E\Big[\mathbf{1}_{A}\,\P(A'\mid (\mathcal{U}_t)_{0\leq t\leq T_0})\times  \E\Big[\mathbf{1}_{A}\Big(\sup_{x\in[u,v]} |\dot \ell^x(\omega_0)-\dot \ell^{-1}(\omega_0)|\Big)^q \,\Big|\,
(\mathcal{U}_t)_{0\leq t\leq T_0}\Big]\Big]
\end{align}
where we use the conditional independence of the excursions of $\mathbf{X}$ given $(X^0_t)_{t\geq 0}$ (equivalently, given $(\mathcal{U}_t)_{0\leq t\leq T_0}$) from \eqref{indtexc}.
From Lemma \ref{maxN*}, one easily verifies that
$$\P(A'\mid (\mathcal{U}_t)_{0\leq t\leq T_0})>0\quad\hbox{a.s.}$$
Furthermore by \eqref{indtexc}, 
$$
 \E\Big[\mathbf{1}_{A}\Big(\sup_{x\in[u,v]} |\dot\ell^x(\omega_0)-\dot \ell^{-1}(\omega_0)|\Big)^q \,\Big|\,(\mathcal{U}_t)_{0\leq t\leq T_0}\Big]
=\mathbf{1}_{A}\Bigg(\frac{1}{2} \N^{*,\Delta_0}\Big(\Big(\sup_{x\in[u,v]} |\dot\ell^x|\Big)^q \Big)
+\frac{1}{2} \check\N^{*,\Delta_0}(|\dot\ell^{-1}(\omega_0)|^q)\Bigg),
$$
and, from \eqref{estim1} and \eqref{estim2}, it follows that
$$\mathbf{1}_{A} \N^{*,\Delta_0}\Big(\Big(\sup_{x\in[u,v]} |\dot\ell^x|\Big)^q \Big)<\infty\quad \hbox{a.s.}$$
Using the conditional distribution of $\Delta_0$ given $A$, we conclude that
$$\N^{*,z}_0\Big(\Big(\sup_{x\in[u,v]} |\dot\ell^x|\Big)^q\Big)<\infty,\hbox{ for a.e. }z>0,$$

We have thus proved that, for a.e. $z>0$,
$$\N^{*,z}_0\Big(\Big(\sup_{x\in[u,v]} |\dot\ell^x|\Big)^q\Big)<\infty, \hbox{ for every } 0<u<v.$$
However, if the last display holds for one value of $z>0$, the scaling in \eqref{N*props}(iii) shows that it must hold for every $z>0$ and 
in fact has a uniform bound for $z\in[1/n,n]$. \endproof

Thanks to the above, the quantity $\N^{*,z}(\dot\ell^a)$ is well defined for every $a>0$ and $z>0$. It can in fact 
be computed explicitly.

\begin{proposition}
\label{expected-derivative}
For every $z>0$ and $a>0$, we have
\begin{equation}
\label{moment1}
\N^{*,z}_0(\ell^a)=  \sqrt{6\pi}\,a^{-2}\,z^{5/2}\,\chi(\frac{3z}{2a^2})
\end{equation}
where, for every $x>0$,
$$\chi(x)=\frac{2}{\sqrt{\pi}} (x^{3/2}+x^{1/2}) -2 x(x+\frac{3}{2})\,e^x\,\mathrm{erfc}(\sqrt{x}),$$
with the notation $\mathrm{erfc}(y)=\frac{2}{\sqrt{\pi}}\int_y^\infty e^{-x^2}\dd x$. Moreover,
for every $z>0$ and $a>0$,
\begin{equation}
\label{moment2}
\N^{*,z}_0(\dot \ell^a)= z\,\gamma\Big(\frac{3z}{2 a^2}\Big)
\end{equation}
where, for every $u>0$,
$$\gamma(u)=-\frac{8}{3}\sqrt{\pi}\,u^{3/2}\,\Big(\chi(u)+u\chi'(u)\Big).$$
\end{proposition}

\rem The function $\chi$
is positive on $(0,\infty)$ and 
its Laplace transform is $\int_0^\infty \chi(z)\,e^{-\lambda z}\,\dd z=(1+\sqrt{\lambda})^{-3}$, cf. the appendix of \cite{Spine}.

\proof  By \cite[Proposition 3]{Spine}, 
 we have, for every nonnegative Borel function $f$ on $[0,\infty)$,
 $$\N^{*,z}_0\Bigg(\int_0^\infty f(a)\,\ell^a\,\dd a\Bigg)=\N^{*,z}_0\Bigg(\int_0^\sigma f(\wh W_s)\,\dd s\Bigg)
 = \int_0^\infty f(a)\,\pi_z(a)\,\dd a$$
 where 
 $$\pi_z(a)= \sqrt{6\pi}\, a^{-2}\,z^{5/2}\,\chi(\frac{3z}{2a^2}),$$
 and $\chi(\cdot)$ is as in the statement. 
 So, we have
\begin{equation}
\label{desint}
\int_0^\infty f(a)\,\N^{*,z}_0(\ell^a)\,\dd a= \int_0^\infty f(a)\,\pi_z(a)\,\dd a.
\end{equation}
It follows that $\N^{*,z}_0(\ell^a)=\pi_z(a)$ for a.e. $a>0$, and Fatou's lemma then gives
$\N^{*,z}_0(\ell^a)\leq\pi_z(a)<\infty$ for every $a>0$. 

If $a>0$ is fixed, we have $\N^{*,z}_0$ a.s.
\begin{equation}
\label{conv-deriv}
\frac{1}{b-a} \int_a^b \dot \ell^c\,\dd c=\frac{1}{b-a} (\ell^b-\ell^a) \build{\la}_{b\to a,b\not =a}^{} \dot \ell^a.
\end{equation}
From Lemma \ref{lem-tech} and dominated convergence, we get that  the convergence \eqref{conv-deriv} holds in $L^1(\N^{*,z}_0)$. 
Consequently,
$$\frac{1}{b-a} (\N^{*,z}_0(\ell^b)-\N^{*,z}_0(\ell^a)) \build{\la}_{b\to a,b\not =a}^{} \N^{*,z}_0(\dot \ell^a).$$

It follows that the function $a\mapsto \N^{*,z}_0(\ell^a)$ is differentiable on
$(0,\infty)$, and 
$$\frac{\dd}{\dd a} \N^{*,z}_0(\ell^a)=\N^{*,z}_0(\dot \ell^a).$$

In particular, since  $a\mapsto \N^{*,z}_0(\ell^a)$ is continuous on $(0,\infty)$, we deduce from \eqref{desint} that
$\N^{*,z}_0(\ell^a)= \pi_z(a)$ for every $a\in(0,\infty)$, which give \eqref{moment1}. 
Then
$$\N^{*,z}_0(\dot \ell^a)= \frac{\dd}{\dd a} \N^{*,z}_0(\ell^a)=\sqrt{6\pi}\Big(-2a^{-3}z^{5/2}\,\chi(\frac{3z}{2a^2})-3\,a^{-5}z^{7/2}\,\chi'(\frac{3z}{2a^2})\Big),$$
and formula \eqref{moment2} follows. \endproof

We now record some asymptotics of the function $\gamma(u)$ introduced in the proposition, which will
be useful in the next sections. We first  note that
\begin{equation}\label{chi'}\chi'(x)= \frac{2}{\sqrt{\pi}}\Big(x^{3/2}+3x^{1/2}+\frac{1}{2}x^{-1/2}\Big)+ \Big(-2x^2-7x-3\Big)\,e^x\mathrm{erfc}(\sqrt{x}),
\end{equation}
and, for every integer $N\geq 0$,
$$e^x\mathrm{erfc}(\sqrt{x})=\frac{1}{\sqrt{\pi}} \sum_{n=0}^N (-1)^{n}\,\frac{1\times 3\times\cdots\times (2n-1)}{2^n}\,x^{-n-1/2}+ O(x^{-N-3/2}),$$
as $x\to\infty$.
By simple calculations it follows that, as $x\to\infty$,
\begin{equation}\label{chiinfty}\chi(x)=\frac{1}{\sqrt{\pi}}\Big(\frac{3}{2} x^{-3/2} -\frac{15}{2} x^{-5/2}+O(x^{-7/2})\Big)
\end{equation}
and 
\begin{equation}\label{chi'infty}\chi'(x)=\frac{1}{\sqrt{\pi}}\Big(-\frac{9}{4}x^{-5/2} +\frac{75}{4}x^{-7/2}+O(x^{-9/2})\Big).
\end{equation}
 Consequently,
\begin{equation}\label{gaminfty}\gamma(x)=2-\frac{30}{x}+O(x^{-2})\  \text{ as $x\to\infty$},
\end{equation}
 and so by \eqref{moment2}, $\N^{*,z}_0(\dot \ell^a)=2z+O(a^{2})$ as $a\to 0$. 
Moreover, from the formulas for $\chi$ and $\chi'$, one has 
\begin{equation}\label{gam0}
\gamma(x)=-8x^2+o(x^2)\text{ as }x\to 0,
\end{equation}
 and therefore $\N^{*,z}_0(\dot \ell^a)=-18a^{-4}\,z^3+o(z^3)$ as $z\to 0$.
 
We can also estimate
\begin{equation}\label{gam'inf}\gamma'(x)=\frac{3}{2}\frac{\gamma(x)}{x} + (-\frac{8}{3}\sqrt{\pi})\,x^{3/2}(2\chi'(x)+x\chi''(x))
=\frac{15}{x} + (-\frac{8}{3}\sqrt{\pi})x^{5/2}\chi''(x)+O(x^{-2}), 
\end{equation}
as $x\to\infty$. Noting that
$$\chi''(x)=\frac{2}{\sqrt{\pi}}\Big(x^{3/2}+5x^{1/2}+3x^{-1/2}-\frac{1}{4} x^{-3/2}\Big)+\Big(-2x^2-11x-10\Big)\,e^x\mathrm{erfc}(\sqrt{x}),$$
we can verify that 
$$x^{5/2}\chi''(x)=\frac{1}{\sqrt{\pi}}\,\frac{45}{8x} + O(x^{-2})$$
and consequently $\gamma'(x)=O(x^{-2})$ as $x\to\infty$. 
From the first equality in \eqref{gam'inf}, \eqref{chi'}, the above expression for $\chi''$, and \eqref{gam0}, one gets that $\gamma'(x)=-16\,x+o(x)$ when $x\to 0$. It follows 
from the preceding estimates for $\gamma'$ that if $\gamma'(0):=0$, then
\begin{equation}\label{gamma'cont}
\gamma'\text{ is continuous on }[0,\infty),
\end{equation} and
\begin{equation}\label{intgam'}
\int_0^\infty |\gamma'(x)|(1\vee x^{-1})\,\dd x<\infty.
\end{equation}

\subsection{The derivative of local times of super-Brownian motion}

We now consider the super-Brownian motion $\mathbf{X}$ started at $\mathbf{X}_0=\alpha\,\delta_0$ constructed as in Section \ref{sec:super}, and its local
times $(L^a)_{a\in\R}$. We fix $a\geq 0$ and $h>0$. Let $\Theta_a$
denote the law of the pair $(L^a,\frac{1}{2}\dot L^a)$ under $\P(\cdot\cap\{L^a>0\})$. Our goal is to compute the
conditional expectation
$$\E\Big[\dot L^{a+h}\,\Big|\, L^a=t,\frac{1}{2}\dot L^a=y\Big]$$
for $t>0$ and $y\in\R$. Recall that we will interpret this conditional expectation as in \eqref{condi-expec}, using \eqref{decomp-local}.  Therefore we can unambiguously make assertions for all $h>0$ simultaneously.

\begin{proposition}
\label{tech-incre}
Let $a\geq 0$. Then, for $\Theta_a$-almost every $(t,y)$, for every $h>0$, we have
$$\E[|\dot L^{a+h}|\,|\, L^a=t,\frac{1}{2}\dot L^a=y]<\infty$$  and
\begin{equation}
\label{key-equa}
\E\Big[\dot L^{a+h}\,\Big|\, L^a=t,\frac{1}{2}\dot L^a=y\Big]=\E\Big[\sum_{j=1}^\infty Z_j\,\gamma\Big(\frac{3Z_j}{2h^2}\Big)\Big],
\end{equation}
where $(Z_j)_{j\geq 1}$ is the sequence of jumps of the $\psi$-L\'evy bridge $U^{\mathrm{br},t,y}$, listed in decreasing order, and
\begin{equation}
\label{bound-series}
\E\Big[\sum_{j=1}^\infty Z_j\,\Big|\gamma\Big(\frac{3Z_j}{2h^2}\Big)\Big|\Big]<\infty, \text{ for every }h>0.
\end{equation}
\end{proposition}

\proof
From the asymptotics derived at the end of Section \ref{sec:casepositive}, we have $|\gamma(z)|\leq C(1\wedge z^2)$ for some constant $C$. Hence, using the
absolute continuity relation \eqref{RNF}, we claim it is easy 
to verify \eqref{bound-series},
so that the right-hand side of \eqref{key-equa} makes sense. To see this, write the sum inside the expectation in \eqref{bound-series} as $S_1+S_2$, where $S_1$ corresponds to the contribution from jumps occurring in $[0,t/2]$ and $S_2$ corresponds to those which occurred in $[t/2,t]$.  Apply \eqref{RNF} to show that $\E[S_1]<\infty$, and its counterpart for the time-reversed process $(y-U^{\text{br},t,y}_{(t-s)-})_{0\le s\le t/2}$ to show $\E[S_2]<\infty$. 

Let $h>0$. By Lemma \ref{moment-deri} (i), $\E[|\dot L^{a+h}-\dot L^{a}|]<\infty$, and therefore we have
\begin{equation}
\label{bound-1mom}
\E[|\dot L^{a+h}|\mid L^a=t,\frac{1}{2}\dot L^a=y]<\infty,\ \text{for $\Theta_a$-a.e.~$(t,y)$}.
\end{equation}
By the convention noted before the Proposition, the quantity $\E[|\dot L^{a+h}|\mid L^a=t,\frac{1}{2}\dot L^a=y]$ is well defined simultaneously for every choice of $t>0$, $y\in\R$, and $h>0$.
 Lemma \ref{moment-deri} (ii) shows that \eqref{bound-1mom} 
holds simultaneously for every $h>0$, for $\Theta_a$-a.e.~$(t,y)$, giving the first required result.
{\bf In what follows, we
fix $t>0$ and $y\in\R$ such that \eqref{bound-1mom} holds for every $h>0$.}

Recall the notation introduced before Proposition \ref{key-ingre}. By \eqref{decomp-local}, we have
\begin{equation}
\label{ident-incre}
\dot L^{a+h}=\sum_{j\in \N} \dot\ell^{a+h}(\omega^{a,+}_j).
\end{equation}
where the sum involves only a finite number of nonzero terms. We know
from Proposition \ref{key-ingre} that the 
conditional distribution of
$(\omega^{a,+}_j)_{j\in\N}$
knowing $L^a=t$ and $\frac{1}{2}\dot L^a=y$,
is the law of $(\varpi_j)_{j\in\N}$, where, conditionally on the (ordered) sequence
$(Z_j)_{j\in\N}$ of jumps of a $\psi$-L\'evy bridge $U^{\mathrm{br},t,y}$, the snake trajectories
$\varpi_j$ are independent and $\varpi_j$ is distributed according to $\N^{*,Z_j}_a$.
Therefore, \eqref{ident-incre} gives
$$\E\Big[\dot L^{a+h}\,\Big|\, L^a=t,\frac{1}{2}\dot L^a=y\Big]=\E\Big[\sum_{j=1}^\infty \dot\ell^{a+h}(\varpi_j)\Big],$$
where $(\varpi_j)_{j\in\N}$ and $(Z_j)_{j\in\N}$ are as described above. Note that the (a.s. finite) sum
$\sum_{j=1}^\infty \dot\ell^{a+h}(\varpi_j)$ is an integrable random variable, as a consequence 
of \eqref{bound-1mom} and \eqref{ident-incre}.
To get \eqref{key-equa} it then suffices to show that
\begin{equation}
\label{tech-incre3}
\E\Big[\sum_{j=1}^\infty \dot\ell^{a+h}(\varpi_j)\Big]=\E\Big[\sum_{j=1}^\infty Z_j\,\gamma\Big(\frac{3Z_j}{2h^2}\Big)\Big].
\end{equation}

For every integer $n\geq 1$, set $N_n:=\max\{j\in\N:Z_j\geq 1/n\}$, with the convention $\max\varnothing=0$. Let $H_n$
stand for the event where $Z_j\leq n$ for every $j\in \N$, and
$W^*(\varpi_j)<a+h$ for every $j>N_n$. 
Then, 
$$
\E\Big[\mathbf{1}_{H_n} \sum_{j=1}^{N_n} | \dot \ell^{a+h}(\varpi_j)|\Big]
=\E\Big[\E\Big[\mathbf{1}_{H_n} \sum_{j=1}^{N_n} |\dot \ell^{a+h}(\varpi_j)|\,\Big|\,(Z_j)_{j\in\N}\Big]\Big]
\leq \E\Big[ \mathbf{1}_{\{Z_j\leq n,\,\forall j\in\N\}} \sum_{j=1}^{N_n} \N^{*,Z_j}_a(|\dot \ell^{a+h}|)\Big]
<\infty
$$
because we know that $\N^{*,z}_a(|\dot \ell^{a+h}|)$ is bounded by a constant if $1/n\leq z\leq n$ (Lemma \ref{lem-tech}), and 
it is easy to verify that $\E[N_n\,|\, L^a=t,\dot L^a=y]<\infty$. For the latter we again may use the absolute continuity of the law of the L\'evy bridge $U^{\mathrm{br},t,y}$ with respect to the law of the L\'evy process
$U$ in \eqref{RNF}, and the analogue for the time-reversed processes, to count the jumps occurring in $[0,t/2]$ and $[t/2,t]$ separately. The preceding display allows us to interchange sum and expected
value in the following calculation,
\begin{align*}
\E\Big[\mathbf{1}_{H_n} \sum_{j=1}^{N_n}  \dot \ell^{a+h}(\varpi_j)\Big]
=\sum_{j=1}^\infty \E\Big[\mathbf{1}_{H_n}\mathbf{1}_{\{j\leq N_n\}} \dot \ell^{a+h}(\varpi_j)\Big]
&=\sum_{j=1}^\infty \E\Big[\mathbf{1}_{H_n}\mathbf{1}_{\{j\leq N_n\}}\,\N^{*,Z_j}_a(\dot \ell^{a+h})\Big]\\
&=\sum_{j=1}^\infty \E\Big[\mathbf{1}_{H_n}\mathbf{1}_{\{j\leq N_n\}}\,Z_j\gamma\Bigl(\frac{3 Z_j}{2 h^2}\Bigr)\Big],
\end{align*}
where \eqref{moment2} is used in the last. In the second equality, we also use the conditional independence of the excursions $\varpi_j$
given their boundary sizes $Z_j$. The left-hand side of the last display is equal to
$$\E\Big[\mathbf{1}_{H_n} \sum_{j=1}^{\infty}  \dot \ell^{a+h}(\varpi_j)\Big]\build{\la}_{n\to\infty}^{} \E\Big[\sum_{j=1}^{\infty}  \dot \ell^{a+h}(\varpi_j)\Big]$$
by dominated convergence (recall that the variable $\sum_{j=1}^{\infty}  \dot \ell^{a+h}(\varpi_j)$ is integrable). On the other hand, the
right-hand side is
$$\E\Big[\mathbf{1}_{H_n}\sum_{j=1}^{N_n}\,Z_j\gamma\Bigl(\frac{3 Z_j}{2 h^2}\Bigr)\Big]\build{\la}_{n\to\infty}^{} \E\Big[\sum_{j=1}^\infty Z_j\gamma\Bigl(\frac{3 Z_j}{2 h^2}\Bigr)\Big]$$
by dominated convergence again, using \eqref{bound-series}.
This completes the proof of \eqref{tech-incre3}, and hence of the proposition.
\endproof

\rem The preceding proof would be much shorter if one could verify that $\E\big[\sum_{j=1}^\infty |\dot\ell^{a+h}(\varpi_j)|\big]<\infty$.
However, this does not seem to follow from our estimates. 

\smallskip
Recall the notation $p_t(y)$ for the density at time $t$ of the L\'evy process $U$ in Section \ref{subsec:bridge}.
To simplify notation, we also set $c_1:=\sqrt{3/8\pi}$, so that the L\'evy measure of $U$ is $\frac{1}{2}\bn(\dd z)=c_1z^{-5/2}\,\mathbf{1}_{(0,\infty)}(z)\,\dd z$. For $h,t>0$ and $y\in\R$, we introduce
\begin{equation}
\label{def-gh}
g_{h}(t,y)=\frac{1}{h}\,\frac{c_1t}{p_t(y)}\,\int_0^{\infty} \Big(p_t(y)-p_t(y-h^2z)\Big)\Big(2-\gamma\Bigl(\frac{3z}{2}\Bigr)\Big)\,\frac{\dd z}{z^{3/2}},
\end{equation}
and set $g_{h}(0,y)=0$. The boundedness of $p_t$, $|p'_t|$ and  $|\gamma|$ (the latter from \eqref{gaminfty} and \eqref{gam0}), and the Mean Value Theorem, show that the above integrand is integrable on $[0,\infty)$.

\begin{proposition}
\label{increment-derivative}
Let $a\geq 0$.
For $\Theta_a$-almost every $(t,y)\in (0,\infty)\times\R$, we have, for every $h>0$,
$$\E\Big[\dot L^{a+h}-\dot L^a\,\Big|\, L^a=t,\frac{1}{2}\dot L^a=y\Big]=g_{h}(t,y),$$
and
$$\lim_{h\to 0} \frac{1}{h}\, \E\Big[\dot L^{a+h}-\dot L^a\,\Big|\, L^a=t,\frac{1}{2}\dot L^a=y\Big]
= 8\,t\, \frac{p'_t(y)}{p_t(y)}.$$
\end{proposition}

\proof {\bf From now on we fix $t>0$ and $y\in \R$} such that \eqref{key-equa} holds for every $h>0$, and we let the sequence 
$(Z_j)_{j\in\N}$ be as in Proposition \ref{tech-incre}. The first statement of Proposition \ref{increment-derivative} will follow
from the computation of 
$$\E\Big[\sum_{j\in\N} Z_j\,\gamma\Big(\frac{3Z_j}{2 h^2}\Big)\Big].$$
We consider the L\'evy process $U$ with Laplace exponent $\psi$ described in the Introduction and Section \ref{subsec:bridge}.
Write $(Y_j)_{j\in \N}$ for the collection of jumps of $U$ over $[0,t]$ (ranked in decreasing size), so that
we have
\begin{equation}
\label{F2}\E\Big[\sum_{j\in\N} Z_j\,\gamma\Big(\frac{3Z_j}{2 h^2}\Big)\Big]=\E\Big[\sum_{j\in\N} Y_j\,\gamma\Big(\frac{3Y_j}{2 h^2}\Big)\,\Big|\, U_t=y\Big].
\end{equation}
(Recall that, when we write $\E[\cdot\mid U_t=y]$, this means that we integrate with respect to the law of the $\psi$-L\'evy bridge from $0$ to $y$ in time $t$.)
We will first compute, for every $\ve >0$,
$$\E\Big[\sum_{j\in\N} Y_j\,\mathbf{1}_{\{Y_j>\ve\}}\,\Big|\,U_t=y\Big].$$
To this end, we evaluate, for every $u\in\R$,
$$\E\Big[\Big(\sum_{j\in\N} Y_j\,\mathbf{1}_{\{Y_j>\ve\}}\Big) e^{\II u U_t}\Big].$$

Set
$$R_\ve= \sum_{j\in\N} Y_j\,\mathbf{1}_{\{Y_j>\ve\}} - 2\,c_1\,t\,\ve^{-1/2}-U_t.$$
The facts that $\E[|U_t|]<\infty$ and $\E[\sum_{j\in\N} Y_j\,\mathbf{1}_{\{Y_j>\ve\}} ]=\frac{t}{2}\int_\ve^\infty x\,\bn(\dd x) <\infty$ imply
\begin{equation}\label{meanRepf}
\E[|R_\ve|]<\infty.
\end{equation}
Recall that $\E[e^{\II u U_t}]= e^{-t\Psi(u)}$, where $\Psi(u)=c_0 |u|^{3/2}\,(1+\II\,\mathrm{sgn}(u))$, with $c_0=1/\sqrt{3}$.
Then 
\begin{equation}
\label{tech}
\E[R_\ve\,e^{\II u U_t}]= \E\Big[\Big(\sum_{j\in\N} Y_j\,\mathbf{1}_{\{Y_j>\ve\}}\Big)e^{\II u U_t}\Big] -2\,c_1\,t\,\ve^{-1/2}\,e^{-t\Psi(u)} - \II \,t\Psi'(u)\,e^{-t\Psi(u)},
\end{equation}
because
$$\E[U_t\,e^{\II u U_t}]= -\II\,\frac{\dd}{\dd u} \E[e^{\II u U_t}]= \II \,t\Psi'(u)\,e^{-t\Psi(u)}.$$
By a classical formula for Poisson measures (Mecke's formula, cf. Theorem 4.1 in \cite{LP}), we have
$$\E\Big[\Big(\sum_{j\in\N} Y_j\,\mathbf{1}_{\{Y_j>\ve\}}\Big)e^{\II u U_t}\Big]=c_1 t\int_\ve^\infty e^{\II uz}\,\frac{\dd z}{z^{3/2}} \times e^{-t\Psi(u)}.$$
Now note that
$$\int_\ve^\infty e^{\II uz}\,\frac{\dd z}{z^{3/2}} = 2\,\ve^{-1/2} - \int_\ve^\infty (1-e^{\II uz})\,\frac{\dd z}{z^{3/2}},$$
and
\begin{align*}
-\int_\ve^\infty (1-e^{\II uz})\,\frac{\dd z}{z^{3/2}}&=-\int_0^\infty (1-e^{\II uz})\,\frac{\dd z}{z^{3/2}} + \int_0^\ve (1-e^{\II uz})\,\frac{\dd z}{z^{3/2}}\\
&=-\sqrt{2\pi}(1-\II\,\mathrm{sgn}(u))\,|u|^{1/2} + \int_0^\ve (1-e^{\II uz})\,\frac{\dd z}{z^{3/2}}.
\end{align*}
On the other hand, since $\Psi'(u)=\frac{3}{2} c_0|u|^{1/2}(1+\II\,\mathrm{sgn}(u))\times\mathrm{sgn}(u)=\frac{3}{2} c_0|u|^{1/2}(\II+\mathrm{sgn}(u))$, we have
$$\II \,t\Psi'(u)=\frac{3}{2} c_0\,t |u|^{1/2}(-1+\II\,\mathrm{sgn}(u))=-c_1\,t\,\sqrt{2\pi} |u|^{1/2}\,(1-\II\,\mathrm{sgn}(u)).$$
By substituting the preceding calculations in \eqref{tech}, we get after simplifications 
\begin{equation}
\label{E1}
\E[R_\ve\,e^{\II u U_t}]= c_1\,t \Big(\int_0^\ve (1-e^{\II uz})\,\frac{\dd z}{z^{3/2}}\Big) \,e^{-t\Psi(u)}.
\end{equation}

Let $\varphi_\ve(x)=\E[R_\ve\mid U_t=x]$ for $x\in\R$. Use \eqref{meanRepf} to see that
\begin{equation}
\label{integra1}
\int_\R |\varphi_\ve(x)|\,p_t(x)\,\dd x\leq \int_\R \E[|R_\ve|\,|\,U_t=x]\,p_t(x)\,\dd x= \E[|R_\ve|] <\infty.
\end{equation}
We have
\begin{equation}
\label{E2}
\E[R_\ve e^{\II uU_t}]=\E[\E[R_\ve\mid U_t]\,e^{\II uU_t}]=\int_\R \varphi_\ve(x)\,p_t(x)\,e^{\II ux}\,\dd x.
\end{equation}
On the other hand, for $0<\delta<\ve$, we can write
$$e^{-t\Psi(u)}\int_\delta^\ve e^{\II uz}\,\frac{\dd z}{z^{3/2}}
=\int_\delta^\ve \Big(\int_\R p_t(x)\,e^{\II ux}\,\dd x\Big)\,e^{\II uz}\,\frac{\dd z}{z^{3/2}}
=\int_\delta^\ve \Big(\int_\R p_t(x-z)\,e^{\II ux}\,\dd x\Big)\,\frac{\dd z}{z^{3/2}},
$$
and
$$ e^{-t\Psi(u)}\!\!\int_\delta^\ve(1-e^{\II uz})\frac{\dd z}{z^{3/2}}
=\int_\delta^\ve \Big(\int_\R (p_t(x)-p_t(x-z))e^{\II ux}\,\dd x\Big)\frac{\dd z}{z^{3/2}}
=\int_\R \Big(\int_\delta^\ve (p_t(x)-p_t(x-z))\frac{\dd z}{z^{3/2}}\Big)e^{\II ux} \dd x.
$$
The last display remains valid for $\delta=0$ as we now show. By dominated convergence to justify the passage to
the limit $\delta\to 0$, it suffices to show 
\begin{equation}
\label{integra2}
\int_\R \Big(\int_0^\ve |p_t(x)-p_t(x-z)|\,\frac{\dd z}{z^{3/2}}\Big) \dd x <\infty.
\end{equation}
For this, use the fact that $x\mapsto p_t(x)$ is unimodal (see Section~\ref{subsec:bridge}) to observe that for $K$ large, for $x\geq K$, and $0\leq z\leq \ve$,
one has $|p_t(x)-p_t(x-z)|= p_t(x-z)-p_t(x)$ and thus 
$$\int_{[K,\infty)} \Big(\int_0^\ve |p_t(x)-p_t(x-z)|\,\frac{\dd z}{z^{3/2}}\Big) \dd x 
= \int_0^\ve \Big(\int_{[K-z,K]} p_t(x)\,\dd x\Big)\frac{\dd z}{z^{3/2}} <\infty$$
because $p_t$ is bounded, argue similarly for $x\leq -K$, and  use the bound $|p_t(x)-p_t(x-z)|\leq Cz$
when 
$-K\leq x\leq K$, where $C$ is a bound for $|p'_t|$. So we have shown \eqref{integra2}, and therefore,
\begin{equation}
\label{E3}
e^{-t\Psi(u)}\int_0^\ve (1-e^{\II uz})\frac{\dd z}{z^{3/2}}=\int_\R \Big(\int_0^\ve (p_t(x)-p_t(x-z))\,\frac{\dd z}{z^{3/2}}\Big)e^{\II ux} \, \dd x.
\end{equation}

From \eqref{E2},\eqref{E1}, and then \eqref{E3}, we get
\begin{equation}
\label{Fourier-equal}
\int_\R \varphi_\ve(x)\,p_t(x)\,e^{\II ux}\,\dd x
= c_1\,t\,\int_\R \Big(\int_0^\ve (p_t(x)-p_t(x-z))\,\frac{\dd z}{z^{3/2}}\Big)e^{\II ux} \, \dd x.
\end{equation}
Observe that both functions $x\mapsto p_t(x)\,\varphi_\ve(x)$ and 
$$x\mapsto \int_0^\ve (p_t(x)-p_t(x-z))\,\frac{\dd z}{z^{3/2}}$$
are integrable with respect to Lebesgue measure and continuous. For the second function, we use \eqref{integra2}  for integrability, and for continuity we apply 
the dominated convergence theorem (with the bound $|p_t(x)-p_t(x-z)|\leq C\,z$). For the first one, we
use \eqref{integra1} for integrability, but we have to check that $\varphi_\ve$
is continuous. This is however easy thanks to the the absolute continuity relation between the L\'evy bridge and
the L\'evy process. In fact, write $t_j$ for the time at which the jump $Y_j$ occurs. Then, we have from \eqref{RNF} that 
$$\E\Bigg[\sum_{j\in\N, t_j\in[0,t/2]} Y_j\mathbf{1}_{\{Y_j>\ve\}}\,\Bigg|\,U_t=x\Bigg]
= \E\Bigg[\Bigg(\sum_{j\in\N, t_j\in[0,t/2]} Y_j\mathbf{1}_{\{Y_j>\ve\}}\Bigg)\frac{p_{t/2}(x-U_{t/2})}{p_t(x)}\Bigg],$$
where the right-hand side is clearly a continuous function of $x$. A time-reversal argument shows the same conclusion if we instead take $t_j\in[t/2,t]$, and the desired continuity property of $\varphi_\ve$ follows. 

From \eqref{Fourier-equal} and the above regularity, we conclude that, for every $x\in\R$,
$$\varphi_\ve(x)=c_1\,t\,\frac{1}{p_t(x)}\int_0^\ve (p_t(x)-p_t(x-z))\,\frac{\dd z}{z^{3/2}}.$$
Therefore, from the definition of $R_\ve$ we have
\begin{equation}
\label{E4}\E\Big[\sum_{j\in \N} Y_j\,\mathbf{1}_{\{Y_j>\ve\}}\,\Big|\,U_t=y\Big]
= y+2\,c_1\,t\,\ve^{-1/2}+ c_1\,t\,\frac{1}{p_t(y)}\int_0^\ve (p_t(y)-p_t(y-z))\,\frac{\dd z}{z^{3/2}}.
\end{equation}
The facts that $\lim_{x\to0}\gamma(x)=0$ and $\gamma'$ is continuous on $[0,\infty)$ (i.e., \eqref{gam0} and \eqref{gamma'cont})  imply
$$\sum_{j\in\N} Y_j\,\gamma\Bigl(\frac{3Y_j}{2h^2}\Bigr)=\sum_{j\in\N} Y_j\int_0^{\frac{3Y_j}{2h^2}} \gamma'(u)\,\dd u
=\int_0^\infty \Big(\sum_{j\in\N} Y_j\,\mathbf{1}_{\{Y_j>2h^2u/3\}}\Big)\,\gamma'(u)\,\dd u,$$
where the interchange between summation and integration holds by the bound $|\gamma'(u)|\leq C\,u$ % integrability of $|\gamma'(u)|u^{-1}$ from \eqref{intgam'},
and the fact that $\P(\sum_{j\in\N} Y_j^2<\infty|U_t=y)=1$.
(The latter again holds for our fixed value of $y$ by the usual Radon-Nikodym argument.) 
From the last display, we get
\begin{equation}
\label{F3}
\E\Big[\sum_{j\in \N} Y_j\,\gamma\Bigl(\frac{3Y_j}{2h^2}\Bigr)\,\Big|\,U_t=y\Big]
= \int_0^\infty \E\Big[\sum_{j\in\N} Y_j\,\mathbf{1}_{\{Y_j>2h^2u/3\}}\,\Big|\, U_t=y\Big]\,\gamma'(u)\,\dd u,
\end{equation}
where now the interchange between expectation and Lebesgue integration is justified by the fact that
$$\int_0^\infty  \E\Big[\sum_{j\in\N} Y_j\,\mathbf{1}_{\{Y_j>2h^2u/3\}}\,\Big|\, U_t=y\Big]\,|\gamma'(u)|\,\dd u<\infty. $$
This holds thanks to \eqref{E4}, the fact that $\int_0^\infty |\gamma'(u)|\,(1\vee u ^{-1/2})\,\dd u<\infty$ (by \eqref{intgam'}), and
\begin{equation}\label{ptincint}
\int_0^\infty |p_t(y)-p_t(y-z)|\,z^{-3/2}\,\dd z<\infty,
\end{equation}
where the last follows from the boundedness of $|p'_t|$.
It follows from \eqref{F3} and \eqref{E4} that
\begin{equation}
\label{F4}\E\Big[\sum_{j\in \N} Y_j\,\gamma\Bigl(\frac{3Y_j}{2h^2}\Bigr)\,\Big|\,U_t=y\Big]=2y + 
\frac{c_1t}{p_t(y)}\int_0^\infty\Bigg(\int_0^{2h^2u/3} (p_t(y)-p_t(y-z))\,\frac{\dd z}{z^{3/2}}\Bigg)\gamma'(u)\dd u,
\end{equation}
where we used the equalities
$$\int_0^\infty \gamma'(u)\,\dd u= \lim_{K\to\infty}(\gamma(K)-\gamma(1/K))=2$$
(by \eqref{gaminfty} and \eqref{gam0}), and 
$$\int_0^\infty \frac{\gamma'(u)}{\sqrt{u}}\,\dd u=0.$$
To get the last equality, first note that 
$$\frac{\gamma'(u)}{\sqrt{u}}= -\frac{8}{3}\sqrt{\pi}\,\Big(\frac{3}{2}(\chi(u)+u\chi'(u))+ u(2\chi'(u)+u\chi''(u))\Big)=  -\frac{8}{3}\sqrt{\pi}\,\frac{\dd }{\dd u}
\Bigl(\frac{3}{2}
u\chi(u)+u^2\chi'(u)\Bigr),$$
and then apply the asymptotics for $\chi$ and $\chi'$ from Section~\ref{sec:casepositive}, namely \eqref{chiinfty} and \eqref{chi'infty}.
Finally, by $\gamma(K)\to 2$ as $K\to\infty$ (by \eqref{gaminfty} again), we have
$$\int_0^\infty\Bigg(\int_0^{2h^2u/3} (p_t(y)-p_t(y-z))\,\frac{\dd z}{z^{3/2}}\Bigg)\gamma'(u)\dd u
=\int_0^{\infty} (p_t(y)-p_t(y-z))\Bigl(2-\gamma\Bigl(\frac{3z}{2h^2}\Bigr)\Bigr)\,\frac{\dd z}{z^{3/2}}.$$
(The interchange of integrals is justified by \eqref{ptincint} and \eqref{intgam'}.)
Insert this into the right-hand side of \eqref{F4}, and then recall
 \eqref{key-equa} and \eqref{F2}, to obtain the explicit formula
\begin{align*}
\E[\dot L^{a+h}-\dot L^a\mid L^a=t,\frac{1}{2}\dot L^a=y]
&= \frac{c_1t}{p_t(y)}\,\int_0^{\infty} (p_t(y)-p_t(y-z))\Bigl(2-\gamma\Bigl(\frac{3z}{2h^2}\Bigr)\Bigr)\,\frac{\dd z}{z^{3/2}}\\
&= \frac{1}{h}\,\frac{c_1t}{p_t(y)}\,\int_0^{\infty} (p_t(y)-p_t(y-h^2z))\Bigl(2-\gamma\Bigl(\frac{3z}{2}\Bigr)\Bigr)\,\frac{\dd z}{z^{3/2}}.
\end{align*}
This gives the first part of the proposition. 
The second part is then immediate from the following elementary lemma.
\qed

Recall the definition of the function $g$ in Theorem \ref{main-th}.

\begin{lemma}\label{lem:gconv}
There is a function $\delta(h)\to 0$ as $h\to 0$, so that for any $K\in\N$ and some constant $C(K)$,
\[\sup_{K^{-1}\leq t\leq K,|y|\le K}\Bigl|\frac{1}{h}g_{h}(t,y)-g(t,y)\Bigr|\le C(K)\delta(h).\]
\end{lemma}
\begin{proof} Note that
\begin{equation}\label{mvt}
\frac{1}{h}g_{h}(t,y)=\frac{c_1t}{p_t(y)}\,\int_0^{\infty} \frac{p_t(y)-p_t(y-h^2z)}{h^2 z}\Big(2-\gamma\Bigl(\frac{3z}{2}\Bigr)\Big)\,\frac{\dd z}{\sqrt z},
\end{equation} 
while \eqref{gaminfty} and \eqref{gam0} imply that
\begin{equation}\label{domint}
\int_0^\infty |2-\gamma\Bigl(\frac{3z}{2}\Bigr)|\,\frac{\dd z}{\sqrt z}<\infty.
\end{equation}
A tedious but straightforward calculation, left for the reader, gives
$$c_1\int_0^\infty \Bigl(2-\gamma\Bigl(\frac{3z}{2}\Bigr)\Bigr)\,\frac{\dd z}{\sqrt{z}}=8.$$
Using the above in \eqref{mvt}, we conclude that
\begin{equation}\label{incbnd1}
\Bigl|\frac{1}{h}g_{h}(t,y)-g(t,y)\Bigr|\le \frac{c_1t}{p_t(y)}\int_0^\infty\Bigl|\frac{p_t(y)-p_t(y-h^2z)}{h^2 z}-p'_t(y)\Bigr|\,|2-\gamma\Bigl(\frac{3z}{2}\Bigr)|\,\frac{\dd z}{\sqrt z}.
\end{equation}
The mean value theorem implies that 
\[\Bigl|\frac{p_t(y)-p_t(y-h^2z)}{h^2 z}-p'_t(y)\Bigr|\le (\Vert p''_t\Vert_\infty h^2 z)\wedge (2\Vert p'_t\Vert_\infty).\]
The boundedness of $|p'_1|$ and $|p''_1|$ (Section 2 of \cite{Zol}) and scaling imply that $\Vert p'_t\Vert_\infty\le c t^{-4/3}$ and $\Vert p''_t\Vert_\infty\le c t^{-2}$. 
Moreover, $p_t(y)$ is bounded below by a positive constant (depending on $K$) when $1/K\leq t\leq K$ and $|y|\leq K$.
 Now use the above bounds in \eqref{incbnd1} to bound the right-hand side of \eqref{incbnd1}, and hence also the left-hand side, for $|y|\le K$ and $1/K\leq t\leq K$ by
\[C(K)\int_0^\infty((h^2z)\wedge 1)|2-\gamma\Bigl(\frac{3z}{2}\Bigr)|\,\frac{\dd z}{\sqrt z}.\]
To complete the proof, define $\delta(h)$ to be the above integral and use \eqref{domint} to see that $\delta(h)\to 0$ as $h\to 0$ by dominated convergence.
\end{proof}

\section{A stochastic differential equation}
\label{sec:SDE}

In this section, we derive the stochastic differential equation satisfied by the process $(L^x,\dot L^x)_{x\geq 0}$.
Recall that for every $t>0$ and $y\in \R$,
$$g(t,y):=8\,t\, \frac{p'_t(y)}{p_t(y)},$$
and $g(0,y)=0$ for every $y\in\R$. Recall also the notation $R$ for the supremum of the support of $\mathbf{Y}$.
By \eqref{Lsupp},  we have $R=\inf\{x\geq 0:L^x=0\}$. 

\begin{lemma}
\label{tech-SDE}
We have $\displaystyle \int_0^R |g(L^x,\frac{1}{2}\dot L^x)|\,\mathbf{1}_{\{g(L^x,\frac{1}{2}\dot L^x)<0\}}\,\dd x <\infty$ a.s.
\end{lemma}

\proof By scaling, we have, for every $t>0$ and $y\in \R$,
\begin{equation}\label{gformu}
g(t,y)=8\,t\, \frac{p'_t(y)}{p_t(y)}= 8\,t^{1/3}\,\frac{p'_1(yt^{-2/3})}{p_1(yt^{-2/3})}.
\end{equation}
The unimodality of the function $p_1$ (Theorem 2.7.5 of \cite{Zol}) shows there is a constant $y_0\in\R$
such that $p'_1(y)\geq 0$ for every $y\leq y_0$.  
%On the other hand, from the
%explicit expression of $p_1$ in terms of the Airy function given in the introduction, one 
%easily verifies that 
%$$\frac{|p'_1(z)|}{p_1(z)} = O(\frac{1}{z}),\quad \hbox{as } z\to+\infty,$$
Recall from \eqref{gbnd1} that  $|p'_1(y)/p_1(y)|$ is bounded above by
a constant $C$ when $y\geq y_0$.
% (note that a similar bound would {\it not} hold
%for $z\leq y_0$). 
Hence, if $g(t,y)<0$ (forcing $p'_1(yt^{-2/3})<0$
and thus $yt^{-2/3}> y_0$), we obtain from the above that $|g(t,y)|\leq 8C\,t^{1/3}$. Finally, we get
$$\int_0^R |g(L^x,\frac{1}{2}\dot L^x)|\,\mathbf{1}_{\{g(L^x,\frac{1}{2}\dot L^x)<0\}}\,\dd x
\leq \int_0^R 8C\, (L^x)^{1/3}\,\dd x<\infty\ \ a.s.,$$
which completes the proof. \endproof

%\begin{theorem}
%\label{SDE-th}
%We can find, possibly on an enlarged probability space, a linear Brownian motion $(B_t)_{t\geq 0}$
%such that, for every $x\geq 0$,
%$$\dot L^x=\dot L^0 + 4\int_0^x \sqrt{L^y}\,\dd B_y + \int_0^x g(L^y,\frac{1}{2}\dot L^y)\,\dd y.$$
%\end{theorem}
%
%\rem The proof will show that $\int_0^\infty |g(L^y,\frac{1}{2}\dot L^y)|\,\dd y<\infty$ a.s., so that the drift term in the equation makes sense
%and defines a finite variation process.
We now turn to the proof of our main result. 

\proof[Proof of Theorem \ref{main-th}]
Let $n\in\N$. 
By Proposition \ref{increment-derivative} (and the known  Markov property of $(L^x,\dot L^x)_{x\geq 0}$), we have for every $u\geq 0$, 
\begin{equation}
\label{Markov1}
\E[\,\dot L^{u+\frac{1}{n}}-\dot L^u\mid (L^r,\dot L^r)_{r\leq u}]
= \E[\,\dot L^{u+\frac{1}{n}}-\dot L^u\mid L^u,\dot L^u]= g_{1/n}(L^u,\frac{1}{2}\dot L^u)\quad\hbox{a.s.}
\end{equation}
Note that the equality of the last display is trivial on the event $\{L^u=0\}=\{u\geq R\}$.

For every real $K>1$, set
\begin{equation}\label{TKdef}T_K:=\inf\{x\geq 0: L^x\vee |\dot L^x| \geq K \hbox{ or }L^x\leq 1/K\},
\end{equation}
and for  every real $a\geq 0$, let $[a]_n$ be the largest number of the 
form $j/n$, $j\in\Z$, smaller than or equal to $a$. Fix $0<s<t$, and let $f$ be
a bounded continuous function on $[0,\infty)\times \R$. We evaluate
$$\mathcal{R}^K_n(s,t):=
\E\Bigg[\Bigg( \dot L^{[t]_n\wedge T_K}- \dot L^{[s]_n\wedge T_K}- \sum_{j=0}^{n[t]_n-n[s]_n-1} \mathbf{1}_{\{[s]_n+\frac{j}{n}< T_K\}}
g_{1/n}\Bigl(L^{[s]_n+j/n},\frac{1}{2}\dot L^{[s]_n+j/n})\Bigr)\Bigg) f(L^{[s]_n},\dot L^{[s]_n})\Bigg],$$
where $g_{1/n}$ is defined in \eqref{def-gh}.
To this end, we observe that
\begin{equation}
\label{tech1} \dot L^{[t]_n\wedge T_K}- \dot L^{[s]_n\wedge T_K}= \sum_{j=0}^{n[t]_n-n[s]_n-1} \mathbf{1}_{\{[s]_n+\frac{j}{n}< T_K\}}
\Big(\dot L^{[s]_n+\frac{j+1}{n}}-\dot L^{[s]_n+\frac{j}{n}}\Big)
-\mathbf{1}_{\{[s]_n\leq T_K<[t]_n\}} \Big(\dot L^{[s]_n+\frac{j_n}{n}}-\dot L^{T_K}\Big)
\end{equation}
where $j_n=\inf\{j\in\Z_+:[s]_n+\frac{j}{n}\geq T_K\}$. Note that, on the event $\{[s]_n\leq T_K<[t]_n\}$, we
have $0\leq [s]_n+\frac{j_n}{n}-T_K\leq \frac{1}{n}$. 

Thanks to \eqref{tech1}, we can rewrite the definition of $\mathcal{R}^K_n(s,t)$ in the form
\begin{align}
\label{tech00}
\mathcal{R}^K_n(s,t)=&\sum_{j=0}^{n[t]_n-n[s]_n-1}\E\Big[\mathbf{1}_{\{[s]_n+\frac{j}{n}< T_K\}}
\Big(\dot L^{[s]_n+\frac{j+1}{n}}-\dot L^{[s]_n+\frac{j}{n}}-g_{1/n}\Bigl(L^{[s]_n+j/n},\frac{1}{2}\dot L^{[s]_n+j/n}\Bigr)\Big)f(L^{[s]_n},\dot L^{[s]_n})\Big]\nonumber\\
&-\E\Big[ \mathbf{1}_{\{[s]_n\leq T_K<[t]_n\}} \Big(\dot L^{[s]_n+\frac{j_n}{n}}-\dot L^{T_K}\Big) f(L^{[s]_n},\dot L^{[s]_n})\Big] .
\end{align}
For every $0\leq j\leq n[t]_n-n[s]_n-1$, \eqref{Markov1} gives
$$\E\Big[ \dot L^{[s]_n+\frac{j+1}{n}}-\dot L^{[s]_n+\frac{j}{n}}\,\Big|\,(L^r,\dot L^r)_{r\leq [s]_n+\frac{j}{n}}\Big]
= g_{1/n}\Bigl(L^{[s]_n+\frac{j}{n}},\frac{1}{2}\dot L^{[s]_n+\frac{j}{n}}\Bigr),$$
so that
$$\E\Big[\Big(\dot L^{[s]_n+\frac{j+1}{n}}-\dot L^{[s]_n+\frac{j}{n}}- g_{1/n}\Bigl(L^{[s]_n+\frac{j}{n}},\frac{1}{2}\dot L^{[s]_n+\frac{j}{n}}\Bigr)\Big) 
\times \mathbf{1}_{\{[s]_n+\frac{j}{n}< T_K\}} f(L^{[s]_n},\dot L^{[s]_n})\Big]=0,$$
and thus, by \eqref{tech00},
$$\mathcal{R}^K_n(s,t) = - \E\Big[ \mathbf{1}_{\{[s]_n\leq T_K<[t]_n\}} \Big(\dot L^{[s]_n+\frac{j_n}{n}}-\dot L^{T_K}\Big) f(L^{[s]_n},\dot L^{[s]_n})\Big].$$
By Lemma \ref{moment-deri}(ii), we have
$$\E\Bigg[\sup_{s/2\leq x<y\leq t+1}\Big(\frac{|\dot L^y-\dot L^x|}{|y-x|^\beta}\Big)\Bigg] <\infty$$
where $\beta>0$. Provided that $n$ is sufficiently large so that $[s]_n>s/2$, we thus get
\begin{equation}
\label{tech2}|\mathcal{R}^K_n(s,t)|\leq C\,n^{-\beta},
\end{equation}
where $C$ is a constant. 
When $n\to\infty$, we have
\begin{equation}
\label{tech3}(\dot L^{[t]_n\wedge T_K},\dot L^{[s]_n\wedge T_K}) \build\la_{}^{\rm a.s.} (\dot L^{t\wedge T_K},\dot L^{s\wedge T_K}),
\end{equation}
and we claim that
\begin{equation}
\label{tech4}\sum_{j=0}^{n[t]_n-n[s]_n-1} \mathbf{1}_{\{[s]_n+\frac{j}{n}< T_K\}}
g_{1/n}\Bigl(L^{[s]_n+j/n},\frac{1}{2}\dot L^{[s]_n+j/n}\Bigr)
\build\la_{}^{\rm a.s.}
\int_{s\wedge T_K}^{t\wedge T_K} g\Bigl(L^u,\frac{1}{2}\dot L^u\Bigr)\,\dd u.
\end{equation}
To justify \eqref{tech4}, note that
$$\sum_{j=0}^{n[t]_n-n[s]_n-1} \mathbf{1}_{\{[s]_n+\frac{j}{n}< T_K\}}
g_{1/n}(L^{[s]_n+j/n},\frac{1}{2}\dot L^{[s]_n+j/n})=\int_{[s]_n}^{[t]_n} n\,g_{1/n}(L^{[r]_n},\frac{1}{2}\dot L^{[r]_n})\,\mathbf{1}_{\{[r]_n< T_K\}}\,\dd r.$$
Lemma~\ref{lem:gconv} implies that $$\lim_{n\to\infty}\sup_{r<T_K}|ng_{1/n}(L^{[r]_n},\frac{1}{2}\dot L^{[r]_n})-g(L^r,\frac{1}{2}\dot L^r)|=0,$$ and \eqref{tech4} now follows.

It follows from \eqref{tech2}, \eqref{tech3}, \eqref{tech4} and the definition of $\mathcal{R}^K_n(s,t)$ (justification is simple because stopping at time $T_K$ makes the dominated convergence theorem easy to apply) that
$$\E\Bigg[\Big(\dot L^{t\wedge T_K}-\dot L^{s\wedge T_K} - \int_{s\wedge T_K}^{t\wedge T_K} g(L^u,\frac{1}{2}\dot L^u)\,\dd u\Bigg)f(L^s,\dot L^s)\Bigg]=0.$$
We have assumed that $s>0$, but clearly we can pass to the limit $s\downarrow 0$ to derive the last display for  $s=0$.
Hence,
$$\dot L^{t\wedge T_K}-\dot L^{0} - \int_{0}^{t\wedge T_K} g(L^u,\frac{1}{2}\dot L^u)\,\dd u$$
is a martingale with respect to the filtration $\mathcal{F}^\circ_t:=\sigma\Big((L^r,\dot L^t)_{r\leq t}\Big)$. 

For $\ve\in(0,1)$, set $S_\ve=\inf\{r\geq 0: L^r\leq \ve\}$. We get that
$$M^\ve_t:=\dot L^{t\wedge S_\ve}-\dot L^{0} - \int_{0}^{t\wedge S_\ve} g(L^u,\frac{1}{2}\dot L^u)\,\dd u$$
is a local martingale (note that, if $R_K:=\inf\{x\geq 0:L^x\vee |\dot L^x|\geq K\}$, $M^\ve_{t\wedge R_K}$
is a martingale, and $R_K\uparrow \infty$ as $K\uparrow \infty$). 

%Furthermore, we know from Proposition \ref{quad-var} that 
We next claim the quadratic variation of $M^\ve$
is
\begin{equation}\label{mepsqv}\langle M^\ve,M^\ve\rangle_t= 16\int_0^{t\wedge S_\ve} L^r\,\dd r.
\end{equation}
To derive this from Proposition~\ref{quad-var}, first fix $t>0$ and let $\pi_n=\{0=t^n_0<t^n_1<\dots<t^n_{m_n}=t\}$ be a sequence of subdivisions of $[0,t]$ such that $\Vert\pi_n\Vert=\max_{1\le i\le m_n}(t^n_i-t^n_{i-1})\to 0$ as $n\to\infty$. 
If $X$ is a stochastic process let $Q(\pi_n,X)=\sum_{i=1}^{m_n}(X(t^n_i)-X(t^n_{i-1}))^2$. Then, taking limits in probability with respect to $\P(\cdot |S_\ve\ge t)$ we have,
\[\langle M^\ve,M^\ve\rangle_t=\lim_{n\to\infty}Q(\pi_n,M^\ve)=\lim_{n\to\infty}Q(\pi_n,\dot L)=16\int_0^tL^r\, \dd r,\]
where we use Proposition~\ref{quad-var} in the last equality, and the fact that $t\le S_\ve$ in the second equality.  
This shows that $\langle M^\ve,M^\ve\rangle_t=16\int_0^{t\wedge S_\ve}L^r\,\dd r$ a.s. on $\{t\le S_\ve\}$ (this conclusion is trivial if this latter set is null, so the implicit assumption above  that it is not null is justified).  By taking left limits through rational values, it follows that 
\[\langle M^\ve,M^\ve \rangle_t=16\int_0^{t\wedge S_\ve}L^r\, \dd r\ \text{ for every }t\le S_\ve\ \text{a.s.}\]
Since $\langle M^\ve,M^\ve \rangle_t$ is constant for $t\ge S_\ve$, \eqref{mepsqv} follows.

If we set
$$\wt B^\ve_t=\int_0^t \frac{1}{4\sqrt{L^r}}\,\dd M^\ve_r$$
then $\wt B^\ve$ is a local martingale with quadratic variation
\begin{equation}\label{Bveqv}\langle\wt B^\ve,\wt B^\ve\rangle_t= t\wedge S_\ve.
\end{equation}
In particular, $\wt B^\ve$ is a (true) martingale. Up to enlarging the
probability space, we can find a linear Brownian motion $B'$ with $B'_0=0$, 
which is independent of $\mathbf{X}$, and thus also of $(L^x,\dot L^x)_{x\in\R}$. 
We introduce the (completion of the) filtration $\f_t:=\f^\circ_t\vee \sigma(B'_r:0\leq r\leq t)$, so that
$\wt B^\ve$ remains a martingale in this filtration. If we set
$$B^\ve_t=\wt B^\ve_t + \int_{t\wedge S_\ve}^t \dd B'_s$$
then one immediately verifies that $B^\ve$ is a martingale of $(\f_t)_{t\geq 0}$ and 
$$\langle B^\ve,B^\ve\rangle_t= t.$$
Therefore $B^\ve$ is a linear Brownian motion. 

Next, suppose that $0<\ve'<\ve<1$. By construction, we have $\wt B^\ve_t=\wt B^{\ve'}_{t\wedge S_\ve}$. We can deduce from 
this that $\wt B^\ve_{S_\ve}$ converges in probability when $\ve\to 0$. Indeed, for every $t>0$,
$$\E[(\wt B^\ve_{S_\ve\wedge t}- \wt B^{\ve'}_{S_{\ve'}\wedge t})^2]= \E[(\wt B^{\ve'}_{S_\ve\wedge t}- \wt B^{\ve'}_{S_{\ve'}\wedge t})^2]
= \E[S_\ve\wedge t-S_{\ve'}\wedge t] \build\la_{\ve,\ve'\to 0, \ve'<\ve}^{} 0,$$
since we know that $S_\ve\uparrow R$ as $\ve\downarrow 0$. 
Let $\Gamma$ stand for the limit in probability of $\wt B^\ve_{S_\ve}$ when $\ve\to 0$. 

Define 
a process $\wt B^0$ by setting $\wt B^0_t=\wt B^\ve_t$ on the event $\{t<S_\ve\}$
(note this does not depend on the choice of $\ve$) and 
$\wt B^0_t=\Gamma$ on the event $\{t\geq R\}$. Finally set
$$B_t:= \wt B^0_{t\wedge R} + \int_{t\wedge R}^t \dd B'_s.$$
Then, it is straightforward to verify that $B^\ve_t$ converges in
probability to $B_t$ when $\ve\to 0$, for every $t\geq 0$ (on the event $\{t\geq R\}$
use the convergence in probability of $\wt B^\ve_{S_\ve}$ to $\Gamma=\wt B^0_R$). The process 
$(B_t)_{t\geq 0}$ has right-continuous sample paths and the same finite-dimensional
marginals as a linear Brownian motion, hence $(B_t)_{t\geq 0}$ is a linear
Brownian motion. More precisely, it is not hard to verify that $(B_t)_{t\geq 0}$ 
is an $(\f_t)$-Brownian motion. 

Next note that
$$M^\ve_t=4\int_0^{t\wedge S_\ve} \sqrt{L^s}\,\dd \wt B^\ve_s=4\int_0^{t\wedge S_\ve} \sqrt{L^s}\,\dd B_s,$$
since $\wt B^\ve_{\cdot\wedge S_\ve}= B^\ve_{\cdot\wedge S_\ve}= B_{\cdot\wedge S_\ve}$. Therefore, we get
\begin{equation}
\label{SDE-approx}
\dot L^{t\wedge S_\ve}= \dot L^0 + 4\int_0^{t\wedge S_\ve} \sqrt{L^s}\,\dd B_s + \int_0^{t\wedge S_\ve} g(L^s,\frac{1}{2}\dot L^s)\,\dd s.
\end{equation}
When $\ve \to 0$, $\dot L^{t\wedge S_\ve}$ converges to $\dot L^{t\wedge R}$ and 
$\int_0^{t\wedge S_\ve} \sqrt{L^s}\,\dd B_s$ converges to $\int_0^{t\wedge R} \sqrt{L^s}\,\dd B_s$ in probability. It 
follows that $\int_0^{t\wedge S_\ve} g(L^s,\frac{1}{2}\dot L^s)\,\dd s$ also converges in probability to a finite random variable.
By Lemma \ref{tech-SDE}, this is only possible if 
$$\int_0^{t\wedge R} g(L^s,\frac{1}{2}\dot L^s)\,\mathbf{1}_{\{g(L^s,\frac{1}{2}\dot L^s)>0\}}\,\dd s <\infty\,\quad\hbox{a.s.},$$
and therefore by the same lemma, 
$$\int_0^{t\wedge R} |g(L^s,\frac{1}{2}\dot L^s)|\,\dd s <\infty\,\quad\hbox{a.s.},$$
which by \eqref{Lsupp} gives the first assertion in Theorem \ref{main-th}. 
We may now let $\ve\to 0$ in \eqref{SDE-approx}, to conclude that
$$\dot L^{t\wedge R}= \dot L^0 + 4\int_0^{t\wedge R} \sqrt{L^s}\,\dd B_s + \int_0^{t\wedge R} g(L^s,\frac{1}{2}\dot L^s)\,\dd s.$$
Since $L^s=\dot L^s=0$ when $s>R$ by \eqref{Lsupp}, this implies the stochastic differential equation \eqref{main-SDE}.% and shows that $(L^{t\wedge R},\dot L^{t\wedge R})=(L^t,\dot L^t)$ for all $t\ge 0$.  

It remains to establish the pathwise uniqueness claim. Let $(X^x,Y^x)$ be any solution to \eqref{main-SDE} such that $(X^0,Y^0)=(L^0,\dot L^0)$ and $(X^x,Y^x)=(X^{R'},Y^{R'})$ for all $x>R'=\inf\{x\ge 0:X^x=0\}$.  
The smoothness of $p_t(y)$ in $(t,y)\in(0,\infty)\times \R$ and strict positivity of $p_t(y)$ for $t>0$ show that $g(t,y)$ is 
Lipschitz  on $[1/K,K]\times [-K,K]$, as is $(t,y)\to \sqrt t$. The classical proof of pathwise uniqueness in It\^o equations with locally Lipschitz coefficients (e.g. Theorem~3.1 in Chapter IV of \cite{IW}) now shows that if $T_K$ is as in \eqref{TKdef} and $T'_K$ is the analogous stopping time for $(X,Y)$, then $T_K=T'_K$ and $(X^{x\wedge T'_K},Y^{x\wedge T'_K})=(L^{x\wedge T_K},\dot L^{x\wedge T_K})$ for all $x\ge 0$ a.s. Then $T'_K=T_K\uparrow R<\infty$ a.s.,
and taking limits along $\{T_K\}$, we see that $R=R'$, $(X^R,Y^R)=(L^R,\dot L^R)=(0,0)$ and $(X^{x\wedge R},Y^{x\wedge R})=(L^{x\wedge R},\dot L^{x\wedge R})$ for all $x\ge 0$ a.s.  It therefore follows that $(X,Y)=(L,\dot L)$ a.s. (both are $(0,0)$ for $x>R$) and the pathwise uniqueness claim is proved.
\endproof

We now show how a transformation of the state space and random time change can reduce the SDE \eqref{main-SDE} to a simple one-dimensional diffusion. We will only use the equation \eqref{main-SDE} and standard stochastic analysis in this discussion.  In particular, we could replace $(L^x,\dot L^x)$ by any solution to \eqref{main-SDE} in $[0,\infty)\times\R$ starting from an arbitrary initial condition in $(0,\infty)\times\R$.  Recall that $R=\inf\{x\geq 0: L^x=0\}$.

\begin{proposition}
\label{simple-SDE}
{\rm(a)} We have 
\begin{equation}
\label{integral-infinite}
\int_0^R (L^x)^{-1/3} \dd x=\infty\quad\hbox{a.s.},
\end{equation}
and therefore can introduce the time change
$$\tau(t)=\inf\{x\geq 0: \int_0^x (L^y)^{-1/3} \dd y\geq t\}<R, \quad t\geq 0.$$
{\rm(b)} Set $Z^x:=\dot L^x(L^x)^{-2/3}$ for every $x\in[0,R)$, and $\tilde Z_t :=Z^{\tau(t)}$ and $\tilde L_t:= L^{\tau(t)}$ for every $t\geq 0$. 
The process $(\wt Z_t,\wt L_t)_{t\geq 0}$ is the pathwise unique solution of the equation
\begin{align}\label{SDE4}
 \tilde Z_t&=\tilde Z_0+4W_t+\int_0^t b(\tilde Z_s)\, \dd s\\
\label{SDE4b} \tilde L_t&=\tilde L_0+\int_0^t\tilde L_s\tilde Z_s\,\dd s,
 \end{align}
where $W$ is a linear Brownian motion, and, for $z\in \R$,
\begin{equation}\label{driftasymp}
b(z):= 8\,\frac{p_1'}{p_1}\Bigl(\frac{z}{2}\Bigl)-\frac{2}{3}\,z^2=-\frac{2}{3}\mathrm{sgn}(z)z^2+O\Bigl(\frac{1}{|z|}\Bigr)\text{ as }z\to\pm\infty.
\end{equation}
{\rm(c)} The process $(\wt Z_t)_{t\geq 0}$ is the pathwise unique solution of \eqref{SDE4} and is a recurrent one-dimensional diffusion process. %Conditional on $\wt Z_0$, 
As $t\to\infty$, $\wt Z_t$ converges weakly (in fact, in total variation) to its unique 
invariant probability measure $\nu(\dd z)=Cp_1(\frac{z}{2})^2 \exp(-\frac{z^3}{36})\,\dd z$, where $C>0$. Moreover, 
 \begin{equation}\label{expsde}
 \tilde L_t=\tilde L_0\exp\Bigl(\int_0^t\tilde Z_s\,\dd s\Bigr)\text{ for all }t\geq 0.
 \end{equation}

\end{proposition}

\rem It is interesting to compare \eqref{integral-infinite} with Hong's results \cite{Hon} showing that
$$\lim_{y\uparrow R}\frac{\log (L^y)}{\log(R-y)}=3\quad\hbox{a.s.}$$

\proof It will be useful to analyze the left tail of $p'_1/p_1$ and so give a counterpart of the $O(1/y)$ right tail behavior in \eqref{rtpt}. One argues
just as before, using the representation in terms of Airy functions (see \eqref{Airatio} and \eqref{rtpt}). In fact the calculation using the asymptotics of $\text{Ai}$ and $\text{Ai}'$, is now easier, but the behavior is quite different:
\begin{equation}\label{ltpt}
\frac{p_1'}{p_1}(y)=\frac{2}{3}y^2+\frac{1}{2y}+O\Bigl(\frac{1}{y^4}\Bigr)\ \ \text{ as }y\to-\infty.
\end{equation}
From \eqref{rtpt} and \eqref{ltpt}, we obtain the asymptotics in \eqref{driftasymp}. 
Then, by \eqref{gformu} we may write \eqref{main-SDE} as
\begin{equation}\label{SDE2}
\begin{aligned}
\dot L^x&=\dot L^0 + 4\int_0^x \sqrt{{L^y}}\,\dd B_y + \int_0^x 8(L^y)^{1/3}\frac{p'_1}{p_1}\Bigl(\frac{Z^y}{2}\Bigr)\,\dd y\\
L^x&=L^0+\int_0^x\dot L^y\,\dd y,
\end{aligned}
\end{equation}
where $Z^x=0$ for $x\geq R$ by convention.
We analyze the above using the coordinates $(Z^x,L^x)$, which by It\^o calculus satisfy for $x<R$,
\begin{equation}\label{SDE3}
\begin{aligned}
Z^x&=Z^0+4\int_0^x(L^y)^{-1/6}\,\dd B_y+\int_0^x(L^y)^{-1/3}\,b(Z^y)\,\dd y\\
L^x&=L^0+\int_0^x(L^y)^{-1/3}L^y Z^y\,dy.
\end{aligned}
\end{equation}
The precise meaning of the above is that it holds for the equation stopped at $R_\ve=\inf\{y\ge0:L^y\le\ve\}$ for all $\ve>0$. We set $\rho:=\int_0^R (L^x)^{-1/3} \dd x$ and now use the random time change
$\tau(t)$ introduced in part (a) of the proposition, observing that this random change makes sense only for $t<\rho$ (at present, we do not yet know that $\rho=\infty$ a.s.). If $\tilde Z_t :=Z^{\tau(t)}$ and $\tilde L_t:= L^{\tau(t)}$
for $t<\rho$, it follows that
 \begin{align}\label{SDE5}
 \tilde Z_t&=\tilde Z_0+4W_t+\int_0^t b(\tilde Z_s)\, \dd s\\
\label{SDE5b} \tilde L_t&=\tilde L_0+\int_0^t\tilde L_s\tilde Z_s\,\dd s,
 \end{align}
 where $W_t=\int_0^{\tau(t)}(L^y)^{-1/6}\, \dd B_y$.
  Again the above equation means that for all $\ve>0$, it holds for the equation stopped at $\rho_\ve:=\tau^{-1}(R_\ve)=\int_0^{R_\ve}(L^x)^{-1/3}\dd x=\inf\{t\ge 0:\tilde L_t\le \ve\}$. Then $W_{t\wedge {\rho_\ve}}=\int_0^{\tau(t\wedge \rho_\ve)}(L^y)^{-1/6}\, \dd B_y$ is a continuous local martingale, with quadratic variation
 \begin{equation}\label{Wqv}
 \langle W_{\cdot\wedge \rho_\ve},W_{\cdot\wedge \rho_\ve}\rangle_t=\int_0^{\tau(t\wedge \rho_\ve)}(L^y)^{-1/3}\, \dd y=t\wedge \rho_\ve.
 \end{equation}
Note that \eqref{SDE5b} implies that \eqref{expsde} holds for $t<\rho$. 

By the same method as in the proof of Theorem \ref{main-th} (compare \eqref{Wqv} with \eqref{Bveqv}), we may assume that $W_t$
is defined for every $t\geq 0$ and is a linear Brownian motion. It follows from the definition of $\rho$ that 
$\liminf_{t\uparrow\rho}\tilde L_t=0$, and therefore
by \eqref{expsde} (for $t<\rho$) we have 
$\liminf_{t\uparrow\rho}\tilde Z_t=-\infty$ a.s. on $\{\rho<\infty\}$. Therefore $\tilde Z$ is the unique solution of \eqref{SDE4} up to its explosion time $\rho\le \infty$ (again use Theorem~3.1 in Chapter~IV of \cite{IW}).  By \eqref{driftasymp} the explosion time of $\tilde Z$ must be infinite a.s. (see Theorem~3.1(1) of Chapter~VI of \cite{IW}). 
We conclude that $\rho=\infty$ a.s., giving part (a)  of the 
proposition, as well as \eqref{expsde} and the fact that $\tilde Z$ is the pathwise unique solution of \eqref{SDE4} in (c). 

The other assertions are now easily derived. Equations \eqref{SDE4} and \eqref{SDE4b} are just \eqref{SDE5} and \eqref{SDE5b}
written for every $t\geq 0$.  Pathwise uniqueness for the system \eqref{SDE4}, \eqref{SDE4b} again follows from Theorem~3.1 in Chapter IV of \cite{IW} by the local Lipschitz nature of the drift coefficient. This completes the proof of (b).

By \eqref{SDE4} above and (2) of Chapter 33 of \cite{Ka}, $\tilde Z$ is a one-dimensional diffusion with scale function 
$$s(x)=\int_0^x\exp\Bigl(-\int_0^y\frac{b(z)}{8}\, \dd z\Bigr)\, \dd y=c\int_0^xp_1\Big(\frac{y}{2}\Big)^{-2}\exp\Bigl(\frac{y^3}{36}\Bigl)\, \dd y,$$
where $c>0$ is a constant. The scale function
maps $\R$ onto $\R$ (as is clear from the above asymptotics for $b$ in \eqref{driftasymp}), and in particular, $\tilde Z$ is a recurrent diffusion (all points are visited w.p. $1$ from every starting point).  
From Chapter~33 of \cite{Ka} (see the discussions prior to Theorem~33.1 and after Theorem~33.9 in \cite{Ka}), the speed measure of the diffusion $s(\tilde Z_t)$ has density $(4s'\circ s^{-1}(y))^{-2}$, and is thus a finite measure  since 
$$\int_\R (s'\circ s^{-1}(y))^{-2}\,\dd y= \int_\R (s'(x))^{-1}\,\dd x <\infty,$$
using \eqref{driftasymp} for the last. 
By Lemmas 33.17 and 33.19 in \cite{Ka}, the diffusion $s(\tilde Z_t)$ has a unique invariant measure which is proportional to its speed measure, and starting at any initial point, will converge weakly to it (in fact in total variation) as $t\to\infty$. Therefore  $\tilde Z_t$ has a unique invariant probability with density proportional to $1/s'(x)$, and will converge to it in the same sense. The proof of (c) is complete.\endproof

 The asymptotics for $p_1$ are $p_1(x)\sim c_-\sqrt{|x|} \exp\Bigl(-\frac{2}{9}|x|^3\Bigr)$ as $x\to-\infty$ and $p_1(x)\sim c_+ |x|^{-5/2}$ as $x\to\infty$, where $c_\pm>0$, and $\sim$ means the ratio approaches $1$ (e.g. \cite{CC} but recall our $p_1$ differs by a scaling constant).  This shows that the invariant density of $\tilde Z$ satisfies
\begin{equation*}
f(x)\sim
\begin{cases}C_-|x|\exp\Bigl(-\frac{|x|^3}{36}\Bigr)&\text{ as $x\to-\infty$}\\
C_+|x|^{-5}\exp\Bigl(-\frac{|x|^3}{36}\Bigr)&\text{ as $x\to+\infty$},
\end{cases}
\end{equation*}
where $C_\pm>0$. 

 In terms of our original local time the weak convergence in (c) means that 
\[\frac{\dot L^{\tau(t)}}{(L^{\tau(t)})^{2/3}}\ \text{converges weakly to }Cp_1\Bigl(\frac{x}{2}\Bigr)^2\exp\Bigl(-\frac{x^3}{36}\Bigr)\dd x\ \text{ as }t\to\infty,\]
where $\tau(t)\uparrow R$ as $t\to\infty$. Again this can be compared with the cubic behavior of $L^x$ near its extinction time from \cite{Hon}. 

Note in the above that $\tau'(t)=\tilde L_t^{1/3}$ is recoverable from $(L^0,\tilde Z)$ by \eqref{expsde}, and so one can reverse the above construction and build $(L^x,\dot L^x)$
from the diffusion $\tilde Z$ and a given initial condition $L^0>0$.  
%We start with 
%an $({\mathcal F}_t)$-Brownian motion $W$ and a  pair of ${\mathcal F}_0$-measurable random variables $(L_0,\tilde Z_0)\in (0,\infty)\times \R$.  Let $\tilde Z_t$, $t\ge 0$ be the pathwise unique solution of \eqref{SDE4'} (we now know there is no explosion) and  define 
%\[\tilde L _t=L^0\exp\Bigl(\int_0^t\tilde Z_s\, \dd s\Bigr)\text{ for }t\ge 0.\] 
The following proposition is immediate from the discussion above and uniqueness in law in \eqref{SDE4}.

\begin{proposition}\label{thm:1ddiffusion}
On a filtered probability space $(\Omega,\f,(\f_t),P)$, let $W$ be an  $({\mathcal F}_t)$-Brownian motion and let $(\Lambda^0,\tilde{\mathfrak{Z}}_0)$ be a  pair of ${\mathcal F}_0$-measurable random variables
with values in $(0,\infty)\times\R$.  There is a pathwise unique solution, $(\tilde{\mathfrak{Z}}_t)_{t\geq 0}$, to $\dd \tilde{\mathfrak{Z}}_t=4\dd W_t + b(\tilde{\mathfrak{Z}}_t)\dd t$ with initial value $\tilde{\mathfrak{Z}}_0$.
For every $t>0$, set
\[\tilde\Lambda_t=\Lambda^0\exp\Bigl(\int_0^t\tilde{\mathfrak{Z}}_s\, \dd s\Big).\] 
Then the following holds. 

\noindent{\rm(a)} $\tilde \Lambda_\infty:=\lim_{t\to\infty}\tilde \Lambda_t=0$, $\lim_{t\to\infty}(\tilde \Lambda_t)^{2/3}\tilde{\mathfrak{Z}}_t=0$, and 
 $R=\int_0^\infty (\tilde \Lambda_s)^{1/3}\,\dd s<\infty$ a.s. 
 
 \noindent{\rm(b)} Introduce the random time change
 \[\int_0^{\sigma(x)} (\tilde \Lambda_s)^{1/3}\,\dd s=x\text{ for }x<R,\text{ and set }\sigma(x)=\infty\text{ for }x\ge R.\]
 Define $\Lambda^x=\tilde \Lambda_{\sigma(x)}$ for $x> 0$ and 
 \begin{equation*}
 \mathfrak{Z}^x=\begin{cases}\tilde{\mathfrak{Z}}_{\sigma(x)}&\text{ if }x<R\\
 0&\text{ if }x\ge R.
 \end{cases}
 \end{equation*}
Then $R=\inf\{x\ge 0:\Lambda^x=0\}$ and $x\mapsto \Lambda^x$ is continuously differentiable on $[0,\infty)$ with derivative $\dot \Lambda^x=\mathfrak{Z}^x(\Lambda^x)^{2/3}$ for $x\ge 0$, where we take the right-hand derivative at $x=0$.

\noindent{\rm(c)} By enlarging our probability space, if necessary, we may assume there is a filtration $({\mathcal G}_x)_{x\ge 0}$ and a $({\mathcal G}_x)$-Brownian motion $(B_x)_{x\geq 0}$ such that $( \Lambda^x,\dot\Lambda^x)_{x\ge 0}$ is the $({\mathcal G}_x)$-adapted solution of \eqref{main-SDE}, stopped at $R$.
\end{proposition}

\end{document}